\documentclass[12pt,a4paper]{amsart} 
\usepackage[utf8]{inputenc}
\usepackage[english]{babel}

\usepackage{amsmath}
\usepackage{amsfonts}
\usepackage{amssymb}
\usepackage{color}

\usepackage{amsthm}

\newtheorem{theorem}{Theorem} [section]

\newtheorem{corollary}[theorem]{Corollary} 
\newtheorem{lemma}[theorem]{Lemma}

\newtheorem{open}{Open Question}

\usepackage{graphics}
\usepackage{epsfig}

\usepackage{url}
\usepackage{hyperref}

\usepackage[a4paper,top=2.5cm,bottom=2.8cm,
left=3.2cm,right=3.2cm]{geometry}
\markleft{\hfill \textsc{Some Geometric Properties of the Yff Points of a Triangle} \hfill}
\newcommand{\degrees}{^\circ}

\long\def\void#1{}
\begin{document}
% ===================
International Journal of  Computer Discovered Mathematics (IJCDM) \\
ISSN 2367-7775 \copyright IJCDM \\
Volume 11, 2026 pp. xx--yy  \\
web: \url{http://www.journal-1.eu/} \\
Received xx Jan. 2026. Published on-line xx Mmm 2026\\ 

\copyright The Author(s) This article is published 
with open access.\footnote{This article is distributed under the terms of the Creative Commons Attribution License which permits any use, distribution, and reproduction in any medium, provided the original author(s) and the source are credited.} \\
% ===========================   
\bigskip
\bigskip

\begin{center}
	{\Large \textbf{Some Geometric Properties of}} \\
	{\Large \textbf{the Yff Points of a Triangle}} \\
	\medskip
	\bigskip
        \bigskip

	\textsc{Stanley Rabinowitz$^a$ and Francisco Javier Garc\'{\i}a Capit\'an$^b$} \\

	$^a$ 545 Elm St Unit 1,  Milford, New Hampshire 03055, USA \\
	e-mail: \href{mailto:stan.rabinowitz@comcast.net}{stan.rabinowitz@comcast.net}%\footnote{Corresponding author} \\
	\\web: \url{http://www.StanleyRabinowitz.com} \\
	
	$^b$ I. E. S. \'Alvarez Cubero, 14800 Priego de C\'ordoba, Spain \\
	e-mail: \href{mailto:garciacapitan@gmail.com}{garciacapitan@gmail.com} \\
	web: \url{http://garciacapitan.epizy.com} \\

\bigskip
%draft revised: Jan. 26, 2025
\bigskip

\end{center}
\bigskip

% ==============================
\textbf{Abstract.}
The Yff points of a triangle were introduced by Peter Yff in 1963.
%in a paper in the \emph{American Mathematical Monthly}.
Since then, very few new facts have been discovered about these points.
We present some geometrical properties of the Yff points of various
shaped triangles which were discovered and proved by computer.



\newenvironment{code}[2]
{
\medskip
\hspace{#1}%
\begin{minipage}{#2}
\color{blue}
}
{
\color{black}
\smallskip
\end{minipage}%
}

\newcommand{\redtext}[1]
{\textcolor{red}{#1}
}

\bigskip
\bigskip
% ================================
% 1 Introduction 
% ================================
%
\section{Introduction}
\label{section:introduction}

The following result was found by Peter Yff in 1963 \cite{Yff}.

\begin{theorem}
\label{thm:yff}
Let $D_1$, $D_2$, $E_1$, $E_2$, $F_1$, $F_2$
be points on the sides of triangle $ABC$ (as shown in Figure~\ref{fig:YffPoints})
such that
$AF_1=BD_1=CE_1=AE_2=BF_2=CD_2=u$
where $u$ is the real root of the equation $x^3=(a-x)(b-x)(c-x)$.
Then $AD_1$, $BE_1$, $CF_1$ meet in a point $Y_1$ and
$AD_2$, $BE_2$, $CF_2$ meet in a point $Y_2$.
\end{theorem}

\begin{figure}[h!t]
\centering
\includegraphics[width=0.8\linewidth]{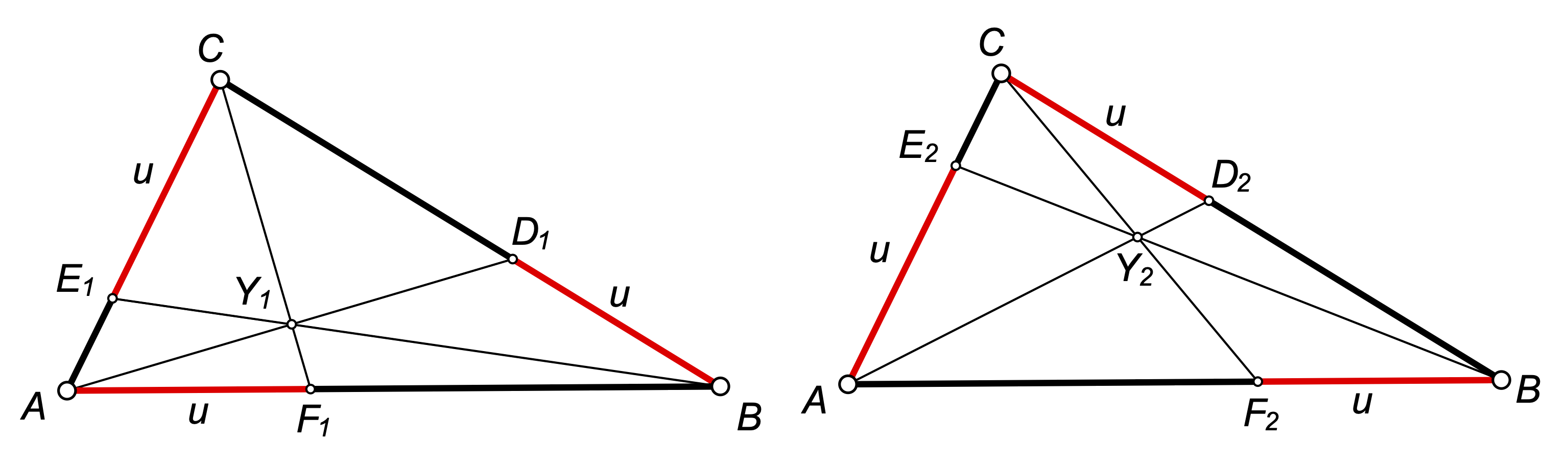}
\caption{Yff Points}
\label{fig:YffPoints}
\end{figure}

The points $Y_1$ and $Y_2$ have become known as the \emph{Yff Points} of a triangle.
If $\triangle ABC$ is named counterclockwise, then $Y_1$ is the \emph{1st Yff point}
and $Y_2$ is the \emph{2nd Yff point}.

Very little is known about these points that was not found in Yff's original paper in 1963.
It is clear from the definition that the Yff points are isotomic conjugates.
A summary of what is known about Yff points can be found in \cite{Yff}, \cite{MathWorld}, \cite[\S 3.3.2]{Yiu},
and \cite{BicentricPairs}.

The value $u$ is connected to $a$, $b$, and $c$ by the equation
\begin{equation}
\label{eq:u}
u^3=(a-u)(b-u)(c-u).
\end{equation}

The value of $u$ can be expressed in terms of radicals as follows.
$$
u=\frac{\sqrt[3]{k_1}}{6 \sqrt[3]{2}}-\frac{k_4}{3\ 2^{2/3} \sqrt[3]{k_1}}+\frac{1}{6} (a+b+c)
$$
where
\begin{align*}
k_1&=k_3+\sqrt{k_2},\\
k_2&=k_3^2 + 4k_4^3,\\
k_3&=2 (a+b+c)^3-18 (ab+bc+ca) (a+b+c)+108 a b c,\\
\intertext{and}
k_4&=6 (ab+ac+bc)-(a+b+c)^2.
\end{align*}

This representation is useful when constructing the Yff points in Dynamic Geometry Enviroments
such as Geometer's Sketchpad or GeoGebra.

Yff also found the following interesting result, where $X_n$ denotes the $n$-th named triangle center
in the Encyclopedia of Triangle Centers \cite{ETC}.

\begin{theorem}
\label{thm:io}
Let $ABC$ be an arbitrary triangle (Figure~\ref{fig:IO}).
Then $Y_1Y_2\perp X_1X_3$.
\end{theorem}

\begin{figure}[h!t]
\centering
\includegraphics[width=0.5\linewidth]{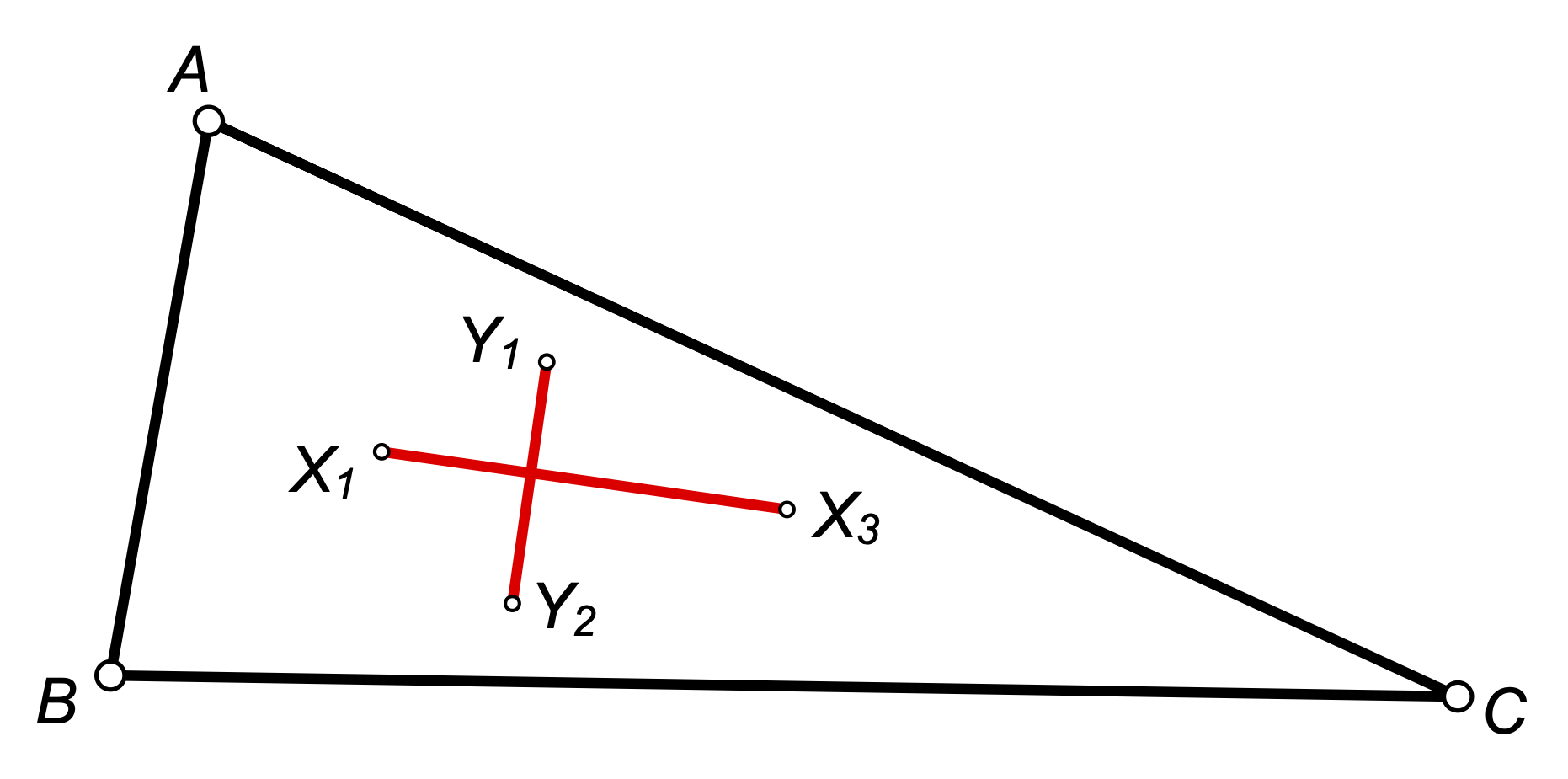}
\caption{red lines are perpendicular}
\label{fig:IO}
\end{figure}

\begin{open}
Is there a purely synthetic geometric proof of Theorem~\ref{thm:io} (without using coordinates)?
\end{open}

In this paper, we will give some new geometric properties of Yff points in various shape triangles.

We found these properties using Mathematica. We started with a random numerical triangle.
As $n$ and $m$ ranged from 1 to 30, we looked at the triangle plus the Yff points plus $X_n$ and $X_m$
and checked (numerically) for various properties in the resulting figure, such as whether two lines were
parallel, two segments had equal length, two angles had equal measure, etc.
We did not find any results that were not equivalent to Theorem~\ref{thm:io}.

We did the same thing for various shaped triangles, such as right triangles, heptagonal triangles,
and triangles whose sides are in arithmetic progression. The results are presented in the next sections.
Since these results were found numerically, valid to 10 decimal places,
the computations did not constitute a proof that the result was true.
At this point, we then used Mathematica again, this time using exact symbolic computations
to prove formally that each result was true.

%\newpage

%\section{Results}
\section{Right Triangles}

\begin{theorem}
\label{thm:X8X20}
Let $ABC$ be a right triangle (named counterclockwise) with right angle at $B$
(Figure~\ref{fig:X8X20}).
Then $AY_2$, $CY_1$, and $X_8X_{20}$ are concurrent.
\end{theorem}

\begin{figure}[h!t]
\centering
\includegraphics[width=0.25\linewidth]{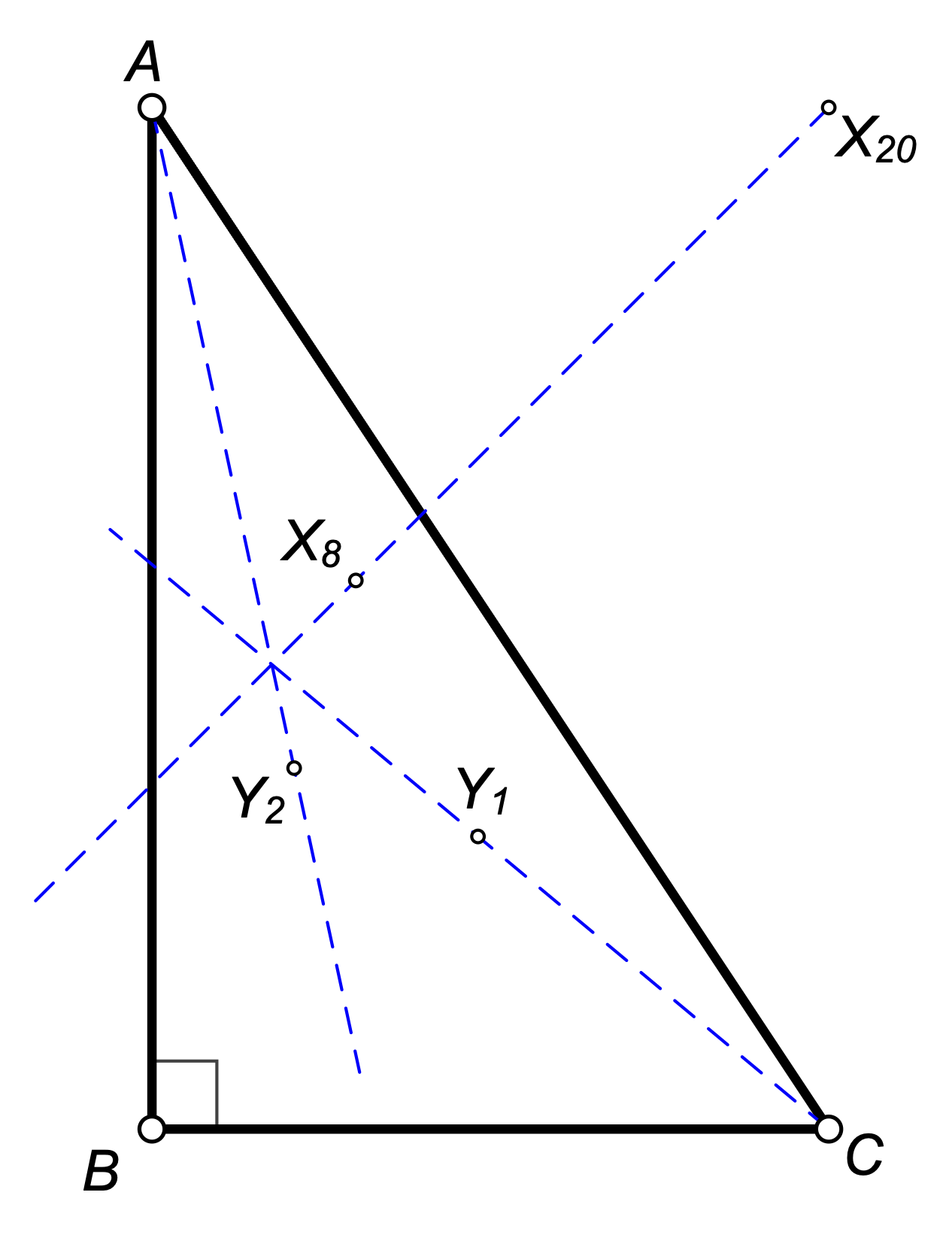}
\caption{dashed lines are concurrent}
\label{fig:X8X20}
\end{figure}

\begin{proof}
Set up a barycentric coordinate system using $\triangle ABC$ as the reference triangle,
so that he coordinates for $A$, $B$, and $C$ are $(1:0:0)$, $(0:1:0)$, and $(0:0:1)$, respectively.
The coordinates for $X_8$ and $X_{20}$ can be found in \cite{ETC}.
\begin{align*}
X_8&=a - b - c::\\
X_{20}&=3 a^4-2 a^2 b^2-2 a^2 c^2-b^4+2 b^2 c^2-c^4::
\end{align*}
where ``$f(a,b,c)$::'' denotes the coordinates $\bigl(f(a,b,c):f(b,c,a):f(c:a:b)\bigr)$.

The coordinates for $Y_1$ and $Y_2$ can be found in \cite{Yff}.
\begin{align*}
Y_1&=\left(\frac{c-u}{b-u}\right)^{1/3}::\\
Y_2&=\left(\frac{b-u}{c-u}\right)^{1/3}::
\end{align*}
where $u$ is given in Theorem~\ref{thm:yff}.
We will call these the \emph{symmetric coordinates} for $Y_1$ and $Y_2$.

Yff also gives simpler coordinates for $Y_1$ that are not as symmetric: $\left(\frac{u}{b-u}:\frac{a-u}{u}:1\right)$.
Eliminating fractions and finding similar coordinates for $Y_2$ give
\begin{equation}
\label{Y1Y2}
\begin{aligned}
Y_1&=\left(u^2:(a-u)(b-u):u(b-u)\right)\\
Y_2&=\left((a-u)(b-u):u^2:u(a-u)\right)
\end{aligned}
\end{equation}
We will call these the \emph{simple coordinates} for $Y_1$ and $Y_2$.

If we represent the line $px+qy+rz=0$ by the triple $(p:q:r)$,
then, we get the nice symmetric properties that the line through two points is
given in Mathematica by the vector cross product \texttt{Cross[point1,point2]} and
the point of intersection of two lines is given by \texttt{Cross[line1,line2]}.
Having found the barycentric coordinates for all the points in Figure~\ref{fig:X8X20},
we can now calculate the equations of the three dashed lines.
To avoid cube roots, we use the coordinates for $Y_1$ and $Y_2$ as given in display (\ref{Y1Y2}) .

From \cite[\S 4.3]{Yiu}, we have that three lines $(p_i:q_i:r_i)$ are concurrent if and only if
$$\left|
\begin{array}{lll}
p_1&q_1&r_1\\
p_2&q_2&r_2\\
p_3&q_3&r_3
\end{array}
\right|=0.$$
In our case, the condition for the three dashed lines to be concurrent is
\void{
\begin{multline}
\label{eq:1}
$$-a^4 b c u+a^4 b u^2+a^4 c u^2-a^3 b^2 c^2+a^3 b^2 c u+a^3 b c^3+a^3 b c^2 u-2 a^3 b c u^2\\
-2 a^3 b u^3-a^3 c^3 u+2 a^3 cu^3-a^2 b^3 c^2+3 a^2 b^3 c u-2 a^2 b^3 u^2-2 a^2 b^2 c u^2\\
+2 a^2 b^2 u^3+a^2 b c^4-2 a^2 b c^3 u+2 a^2 b c^2 u^2-a^2c^4 u+2 a^2 c^3 u^2-2 a^2 c^2 u^3\\
+a b^4 c^2-a b^4 c u+a b^3 c^3-3 a b^3 c^2 u+2 a b^3 c u^2+2 a b^3 u^3-a b^2 c^4+2 ab^2 c^3 u\\
-2 a b^2 c u^3-a b c^5+a b c^4 u-2 a b c^3 u^2+2 a b c^2 u^3+a c^5 u-2 a c^3 u^3+b^5 c^2\\
-2 b^5 c u+b^5 u^2-2b^4 c^3+2 b^4 c^2 u+b^4 c u^2-2 b^4 u^3+b^3 c^3 u-2 b^3 c^2 u^2+2 b^2 c^5\\
-3 b^2 c^4 u+2 b^2 c^3 u^2-b c^6+b c^5 u+b c^4u^2+c^6 u-3 c^5 u^2+2 c^4 u^3=0.
\end{multline}
}
\begin{equation}
\label{eq:2}
\begin{aligned}
&-2 a^7 b u+2 a^7 u^2+2 a^6 b^2 u-4 a^6 b c u+4 a^6 b u^2+4 a^6 c u^2-6 a^6 u^3+4 a^5 b^3 u\\
&-4 a^5 b^2u^2+6 a^5 b c u^2-4 a^5 b u^3-6 a^5 c u^3+4 a^5 u^4-4 a^4 b^4 u-4 a^4 b^3 u^2\\
&+6 a^4 b^2 u^3+4 a^4 bc^3 u-4 a^4 b c^2 u^2-2 a^4 b c u^3+4 a^4 b u^4-4 a^4 c^3 u^2+4 a^4 c^2 u^3\\
&-2 a^3 b^5 u+6 a^3 b^4u^2+2 a^3 b c^4 u-4 a^3 b c^3 u^2+4 a^3 b c^2 u^3-2 a^3 c^4 u^2+4 a^3 c^3 u^3\\
&-8 a^3 c^2 u^4+2 a^2 b^6u+4 a^2 b^5 c u-4 a^2 b^4 c u^2-2 a^2 b^4 u^3-4 a^2 b^3 c^3 u-2 a^2 b^2 c^4 u\\
&+4 a^2 b^2 c^3 u^2+2 a^2c^4 u^3-4 a b^6 u^2-6 a b^5 c u^2+4 a b^5 u^3+4 a b^4 c^2 u^2+6 a b^4 c u^3\\
&-4 a b^4 u^4+8 a b^3 c^3u^2-4 a b^3 c^2 u^3-8 a b^2 c^3 u^3-2 a b c^5 u^2+2 a c^5 u^3+4 a c^4 u^4\\
&+2 b^6 u^3+2 b^5 c u^3-4 b^5u^4-4 b^4 c^2 u^3-4 b^3 c^3 u^3+8 b^3 c^2 u^4+2 b^2 c^4 u^3+2 b c^5 u^3\\
&-4 b c^4 u^4=0
\end{aligned}
\end{equation}
Note that the radical form for $u$ is too complex, so we have left $u$ as an unevaluated variable,
noting that $u$ satisfies Equation~(\ref{eq:u}).

We can eliminate $u$ from equations (\ref{eq:2}) and (\ref{eq:u}).
This can be accomplished using elimination theory and Gr\"obner bases.
The process is straightforward, but tedious and is best left for computers.
In Mathematica, the \texttt{Eliminate} function is used for this purpose.
Eliminating $u$ from equations (\ref{eq:2}) and (\ref{eq:u}) gives
\void{
$$a^6 b^3 c^4+a^5 b^5 c^3-4 a^5 b^4 c^4+a^5 b^3 c^5-2 a^4 b^6 c^3+3 a^4 b^5 c^4+4 a^3 b^6 c^4-4 a^3 b^5 c^5$$
$$+2 a^2 b^8 c^3-5   a^2 b^7 c^4+4 a^2 b^5 c^6-a^2 b^3 c^8-a b^9 c^3+5 a b^7 c^5-4 a b^6 c^6-3 a b^5 c^7$$
$$+4 a b^4 c^8-a b^3 c^9+b^9 c^4-2 b^8c^5+2 b^6 c^7-b^5 c^8=0$$
}
$$-a^{11} b^4c^2-a^{10} b^6 c+3 a^{10} b^5 c^2-2 a^{10} b^4 c^3+a^9 b^7 c+3 a^9 b^5 c^3-a^9
   b^4 c^4+2 a^8 b^8 c$$
$$-5 a^8 b^7 c^2+a^8 b^6 c^3-2 a^7 b^9 c+a^7 b^8 c^2+a^7 b^4 c^6-a^6 b^{10} c+3 a^6
   b^9 c^2-a^6 b^6 c^5$$
$$-3 a^6 b^5 c^6+2 a^6 b^4 c^7+a^5 b^{11} c-3 a^5 b^9 c^3-a^5 b^8 c^4+5 a^5 b^7 c^5-3
   a^5 b^5 c^7+a^5 b^4 c^8$$
$$-a^4 b^{11} c^2+a^4 b^{10} c^3+2 a^4 b^9 c^4-2 a^4 b^8 c^5-a^4 b^7 c^6+a^4 b^6
   c^7=0$$
Polynomials of degree 17 in three variables are not easy to factor, but using the \texttt{Factor}
function in Mathematica, we find that this condition is equivalent to
%$$b^3 c^3 (a-c) (a-b+c) \left(a^2-b^2+c^2\right) \left(a^2 c+a b^2-3 a b c+a c^2-b^3+b^2 c\right)=0.$$
$$a^4 b^4 c (a-c) (a-b+c) (a+b+c) (a^2-b^2+c^2 )(a^2 c+a b^2-3 a b c+a c^2-b^3+b^2c)=0.$$
The factor $a^2-b^2+c^2$ shows that this condition will be true for a right triangle with vertex at $B$.
\end{proof}

This proof is not satisfying.
The computation involved in this proof is enormous. It seems highly unlikely that Theorem~\ref{thm:X8X20}
is true merely because the determinant condition for collinearity happens to evaluate to 0.
The points $Y_1$ and $Y_2$ have nice geometric definitions as shown in Figure~\ref{fig:YffPoints}.
The points $X_8$ and $X_{20}$ also have nice geometric definitions.
(Point $X_8$ is the Nagel Point of a triangle and Point $X_{20}$, also known as the de Longchamps Point of a triangle, is the
reflection of the orthocenter about the circumcenter.)

\begin{open}
Is there a purely synthetic geometric proof of Theorem~\ref{thm:X8X20}?
\end{open}

\void{
\begin{theorem}
Let $ABC$ be a right triangle (named counterclockwise) with right angle at $B$
(Figure~\ref{fig:rightTriangle-X8X22}).
Then $AY_2$, $CY_1$, and $X_8X_{22}$ are concurrent.
\end{theorem}

\begin{figure}[h!t]
\centering
\includegraphics[width=0.5\linewidth]{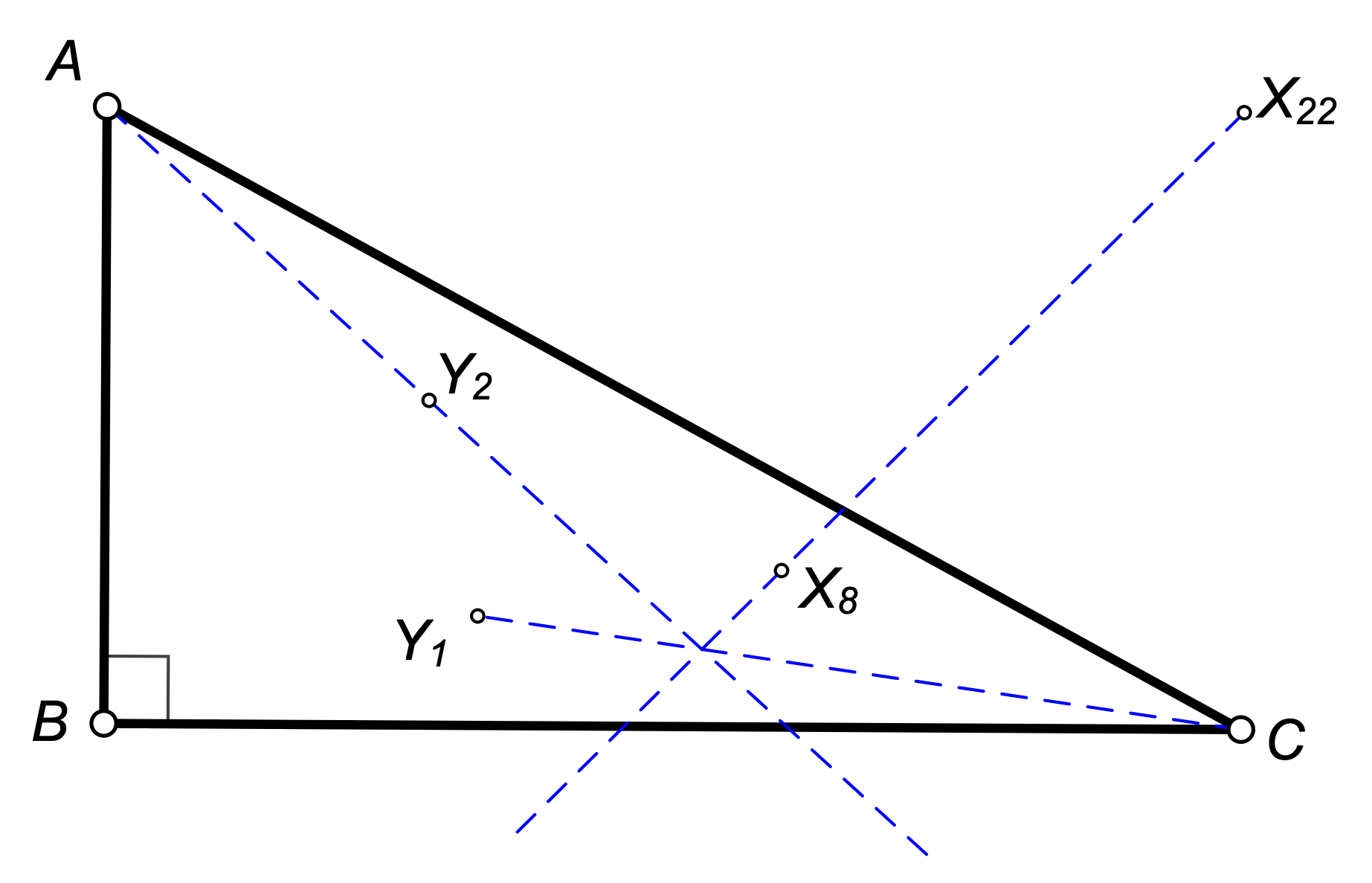}
\caption{dashed lines are concurrent}
\label{fig:rightTriangle-X8X22}
\end{figure}
}

\begin{theorem}
Let $ABC$ be a right triangle (named counterclockwise) with right angle at $B$
and $\angle ACB=30\degrees$ (Figure~\ref{fig:30-60-90-X8X21}).
Then $AY_1$, $BY_2$, and $X_8X_{21}$ are concurrent.
\end{theorem}

\begin{figure}[h!t]
\centering
\includegraphics[width=0.4\linewidth]{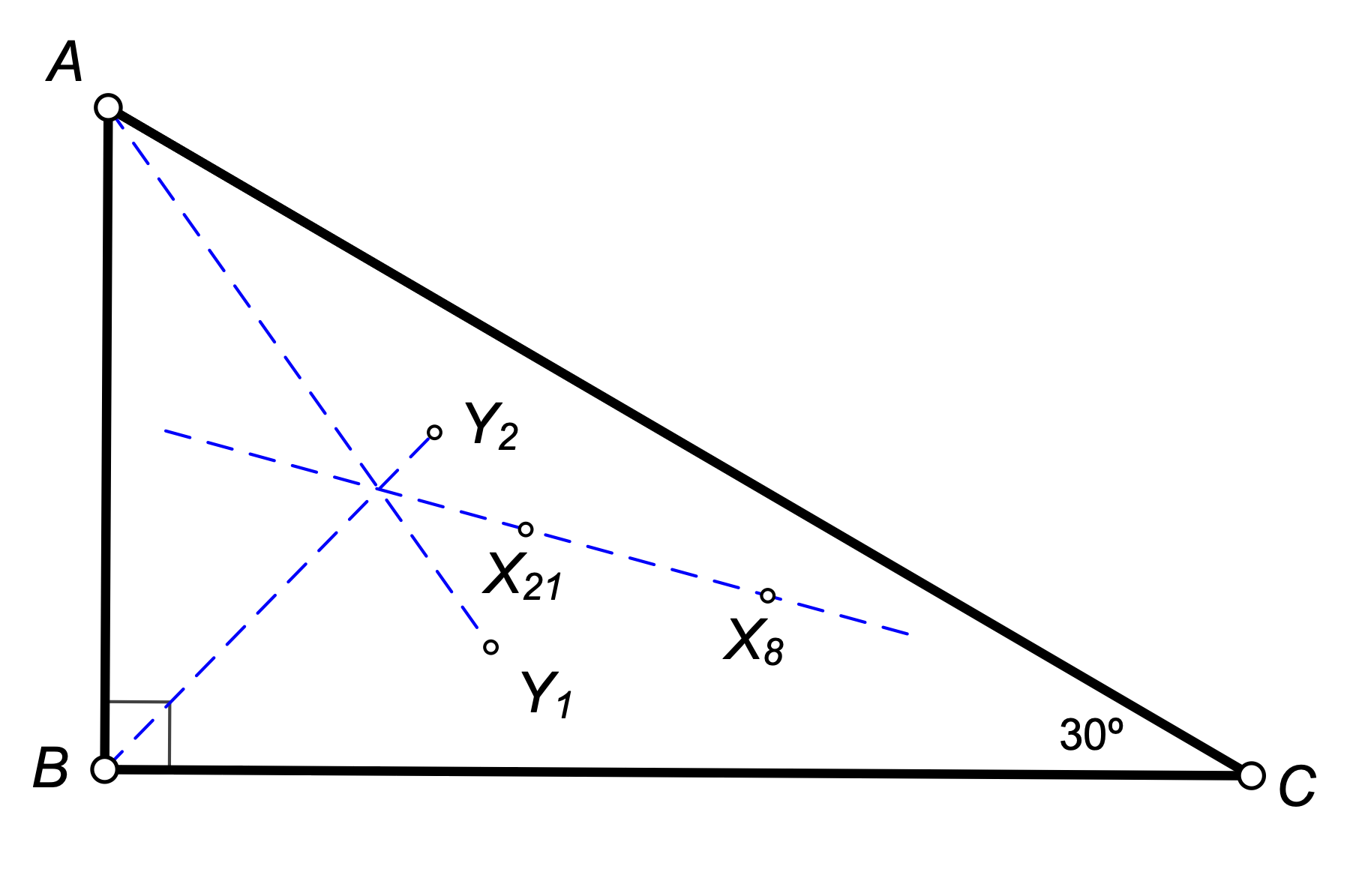}
\caption{dashed lines are concurrent}
\label{fig:30-60-90-X8X21}
\end{figure}

\begin{proof}
Using the same procedure as in the proof of Theorem~\ref{thm:X8X20},
we find that the condition for concurrence is
$(a - b) (a^4 + a^3 b - 2 a^2 c (b + c) + a b (b^2 - 2 b c - 3 c^2) + (b^2 - c^2)^2) (a + b - c - 2 u) (a - u) (b - u) u=0.$
Eliminating $u$ from this equation and Equation (\ref{eq:u}) gives
$a^4 (a - b) b^4 c (a + b + c) (a^2 b + a b^2 - 3 a b c + a c^2 + 
    b c^2 - c^3) (a^3 + b^3 - a^2 c - a b c - b^2 c - a c^2 - b c^2 +  c^3) = 0.$
Setting $b=\sqrt{a^2+c^2}$ and $a=c\sqrt3$ shows that this condition is equivalent to $0=0$ which is always true.
\end{proof}

The following result comes from \cite[\S4.5]{Yiu}.

\begin{lemma}
\label{lemma:inf}
The infinite point on the line represented by $(p:q:r)$ is $(q - r: r - p: p - q)$.
\end{lemma}

\begin{theorem}
Let $ABC$ be a right triangle with right angle at $B$
(Figure~\ref{fig:BX8}).
Then $BX_8\perp Y_1Y_2$.
\end{theorem}

\begin{figure}[h!t]
\centering
\includegraphics[width=0.2\linewidth]{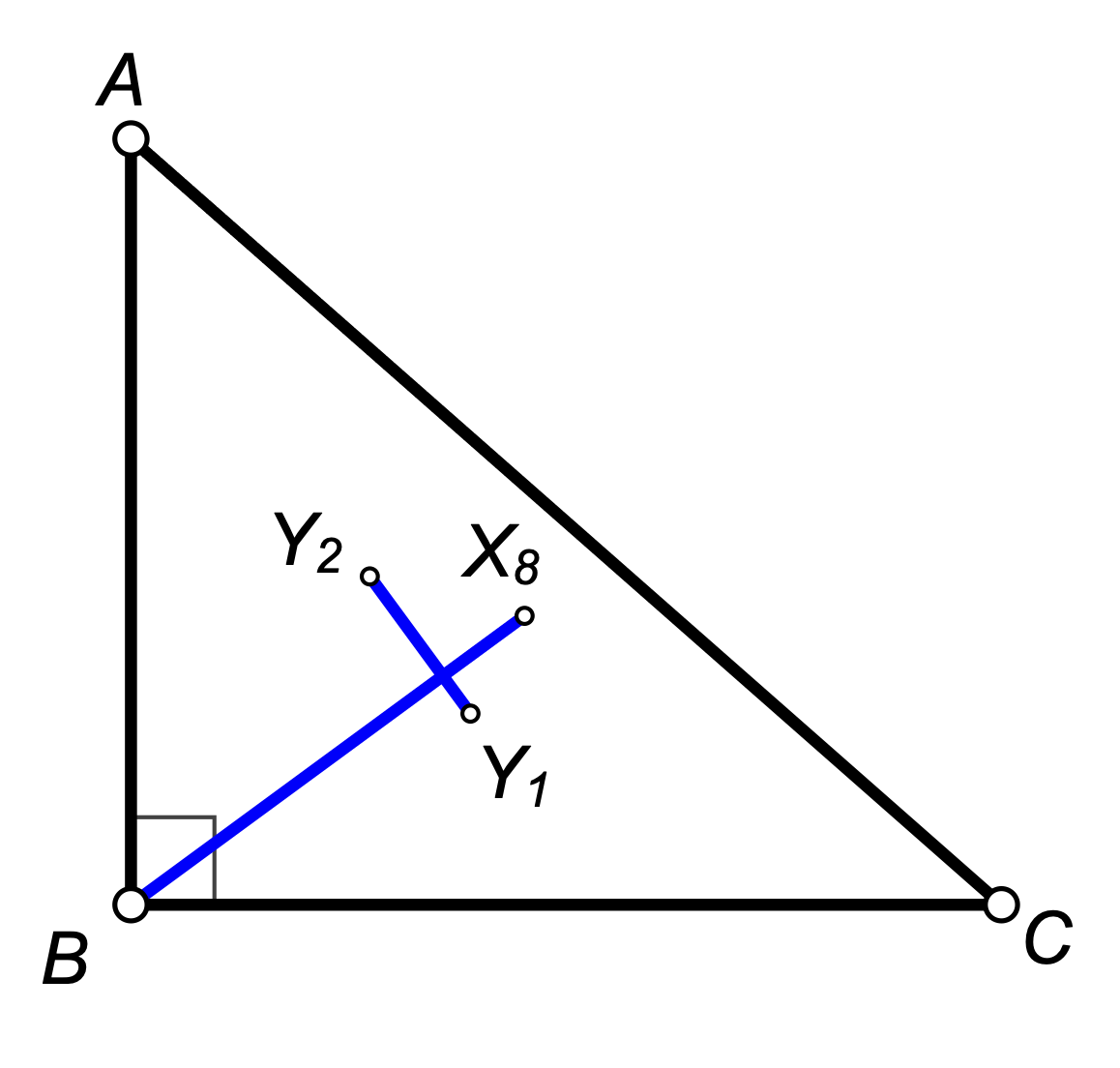}
\caption{blue lines are perpendicular}
\label{fig:BX8}
\end{figure}

\begin{proof}
From \cite{ETC}, the barycentric coordinates for $X_8$ are $(a - b - c:-a + b - c:-a - b + c)$.
The barycentric coordinates for $B$ are $(0:1:0)$.
The line $BX_8$ is represented by \texttt{Cross}$[B,X_8]=(a+b-c:0:a-b-c)$.
Using Lemma~\ref{lemma:inf}, we find that the infinite point on the line $BX_8$ has coordinates $(b+c-a:-2b:a+b-c)$.

Now the coordinates for $X_1$ are $(a:b:c)$. The circumcenter, $X_3$, is the midpoint of
the hypotenuse in a right triangle, so its coordinates are $(1:0:1)$.
The line $X_1X_3$ is therefore represented by \texttt{Cross}$[X_1,X_3]=(b:c-a:-b)$.
So, the infinite point on the line $X_1X_3$ in a right triangle is $(b+c-a:-2b:a+b-c)$.

Since $BX_8$ and $X_1X_3$ have the same point at infinity, they must be parallel.
But $X_1X_3\perp Y_1Y_2$ by Theorem~\ref{thm:io}. Therefore, $BX_8\perp Y_1Y_2$.
\end{proof}

%\newpage
\section{$60\degrees$ Triangles}

\begin{theorem}
Let $ABC$ be a triangle (named counterclockwise) with $\angle C=60\degrees$
(Figure~\ref{fig:60X8}).
Then $AY_1$, $BY_2$, and $X_3X_8$ are concurrent.
\end{theorem}

\begin{figure}[h!t]
\centering
\includegraphics[width=0.4\linewidth]{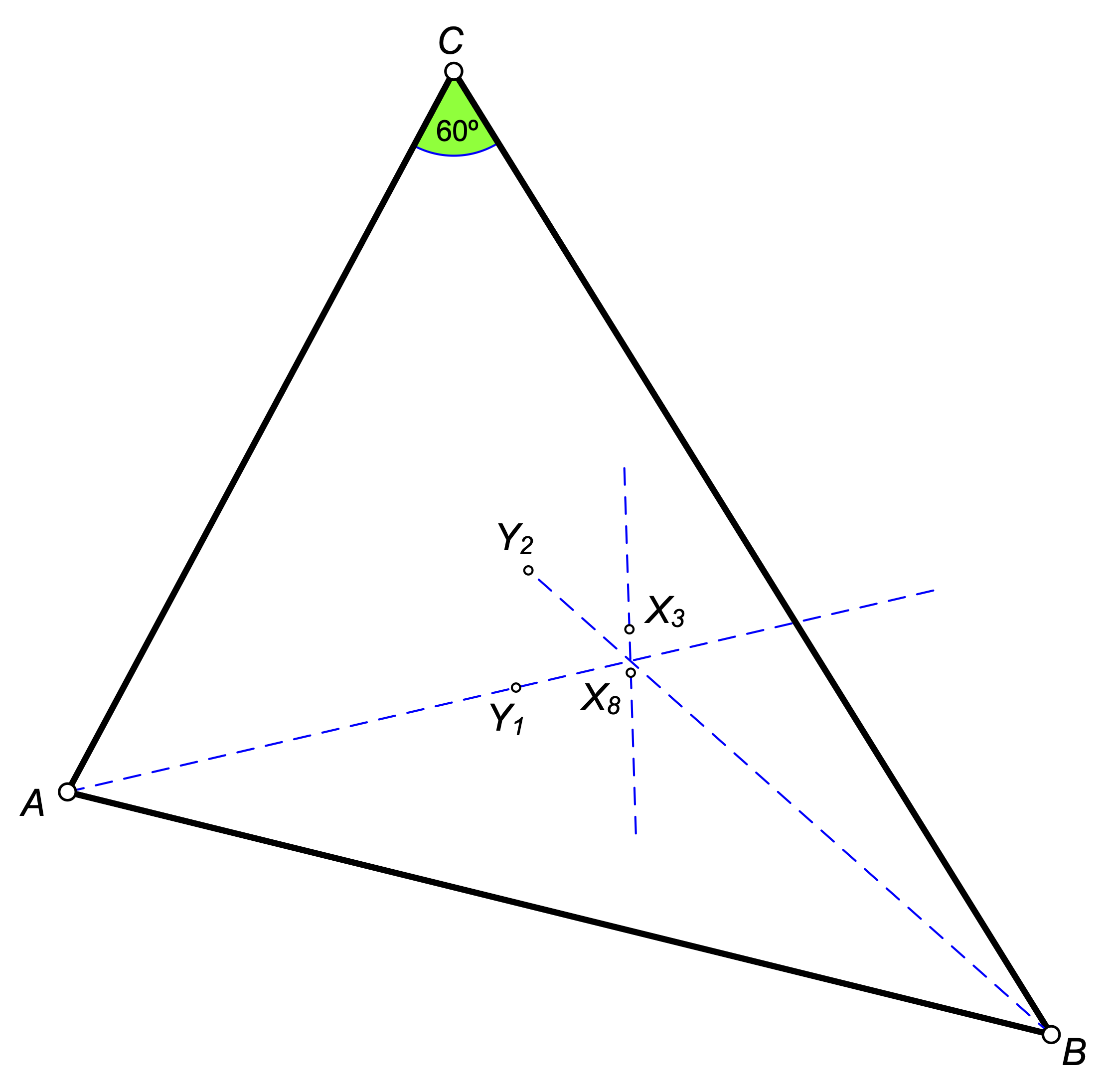}
\caption{dashed lines are concurrent}
\label{fig:60X8}
\end{figure}

\begin{proof}
As before, we find that the condition for concurrence is
$$(a - b) (a + b - c) (a^2 - a b + b^2 - c^2) (a + b - c - 2 u) u=0.$$
From the Law of Cosines, we see that in this triangle, $c^2=a^2  + b^2- a b $, so this condition holds.
\end{proof}

%\newpage

\section{Heptagonal Triangles}

A \emph{heptagonal triangle} is a triangle whose angles are $\pi/7$, $2\pi/7$, and $4\pi/7$.

\begin{lemma}
\label{lemma:heptagonal}
Let $ABC$ be a heptagonal triangle with $a>b>c$. Then
$$a:b:c=\sin\frac{4\pi}{7}:\sin\frac{2\pi}{7}:\sin\frac{\pi}{7}.$$
\end{lemma}

\begin{proof}
This follows immediately from the Law of Sines.
\end{proof}

\begin{lemma}[Reflection About BC]
\label{lemma:BCreflect}
Let $P$ have barycentric coordinates $(p:q:r)$ with respect to $\triangle ABC$.
Then $Q$, the reflection of $P$ about side $BC$, has coordinates
$$\left(a^2 p: (c^2 - b^2) p - a^2 (p + q): (b^2 - c^2) p - a^2 (p + r)\right).$$
\end{lemma}

\begin{proof}
The formula for the distance between two points in barycentric coordinates
can be found in \cite[\S7.1]{Yiu}. Let $Q=(p':q':r')$.
Solving the two equations $PB=QB$ and $PC=QC$ for $p'$, $q'$, and $r'$
and excluding the solution where $Q=P$ gives the required coordinates for $Q$.
\end{proof}

Another proof, not using the distance formula, is instructive.

\begin{proof}
It is immediate that infinite point of the line through the points $(p_1:q_1:r_1)$ and $(p_2:q_2:r_2)$ with 
the same sum $p_1+q_1+r_1=p_2+q_2+r_2$ is given by $(p_1-p_2:q_1-q_2:r_1-r_2)$. Thus, if $(p,q,r)$ is a given
point, we can consider the point $X=(0: q S_B + (p + q) S_C: (p + r) S_B + r S_C)$ with 
sum $(p+q+r)(S_B+S_C)=(p+q+r)a^2$. Then, by subtraction
\[\begin{aligned}
   & \left( {0:q{S_B} + (p + q){S_C}:r{S_C} + (p + r){S_B}} \right) - {a^2}\left( {p:q:r} \right) \\ 
   =  & \left( { - {a^2}p:q{S_B} + (p + q){S_C} - {a^2}q:r{S_C} + (p + r){S_B} - {a^2}r} \right) \\ 
   =  & \left( { - {a^2}p:p{S_C}:p{S_B}} \right),
\end{aligned} \]
that is, the infinite point of $PX$ is the infinite point of the $A-$altitude and 
then $PX$ is perpendicular to $BC$. 

Again, by subtraction,
\[\begin{aligned}
   & 2\left( {0:q{S_B} + (p + q){S_C}:r{S_C} + (p + r){S_B}} \right) - {a^2}\left( {p:q:r} \right) \\ 
   =  & \left( { - {a^2}p:q \cdot 2{S_B} + (p + q) \cdot 2{S_C} - {a^2}q:r \cdot 2{S_C} + (p + r) \cdot 2{S_B} - {a^2}r} \right) \\ 
   =  & \left( { - {a^2}p:2{S_C}p + {a^2}q:2{S_B}p + {a^2}r} \right), \\ 
\end{aligned} \] 
therefore the reflection of $P$ in $X$ is the point 
\[\left( { - {a^2}p:({b^2} - {c^2})p + (p + q){a^2}:({c^2} - {b^2})p + (p + r){a^2}} \right).\qedhere\]
\end{proof}

\begin{theorem}
\label{thm:HeptaY1reflect}
Let $ABC$ be a heptagonal triangle (named counterclockwise) with $AB<AC<BC$
(Figure~\ref{fig:HeptaY1reflect}).
Let $Y_1'$ be the reflection of $Y_1$ about side $BC$.
Then $X_1Y_2\parallel Y_1'C$.
\end{theorem}

\begin{figure}[h!t]
\centering
\includegraphics[width=0.5\linewidth]{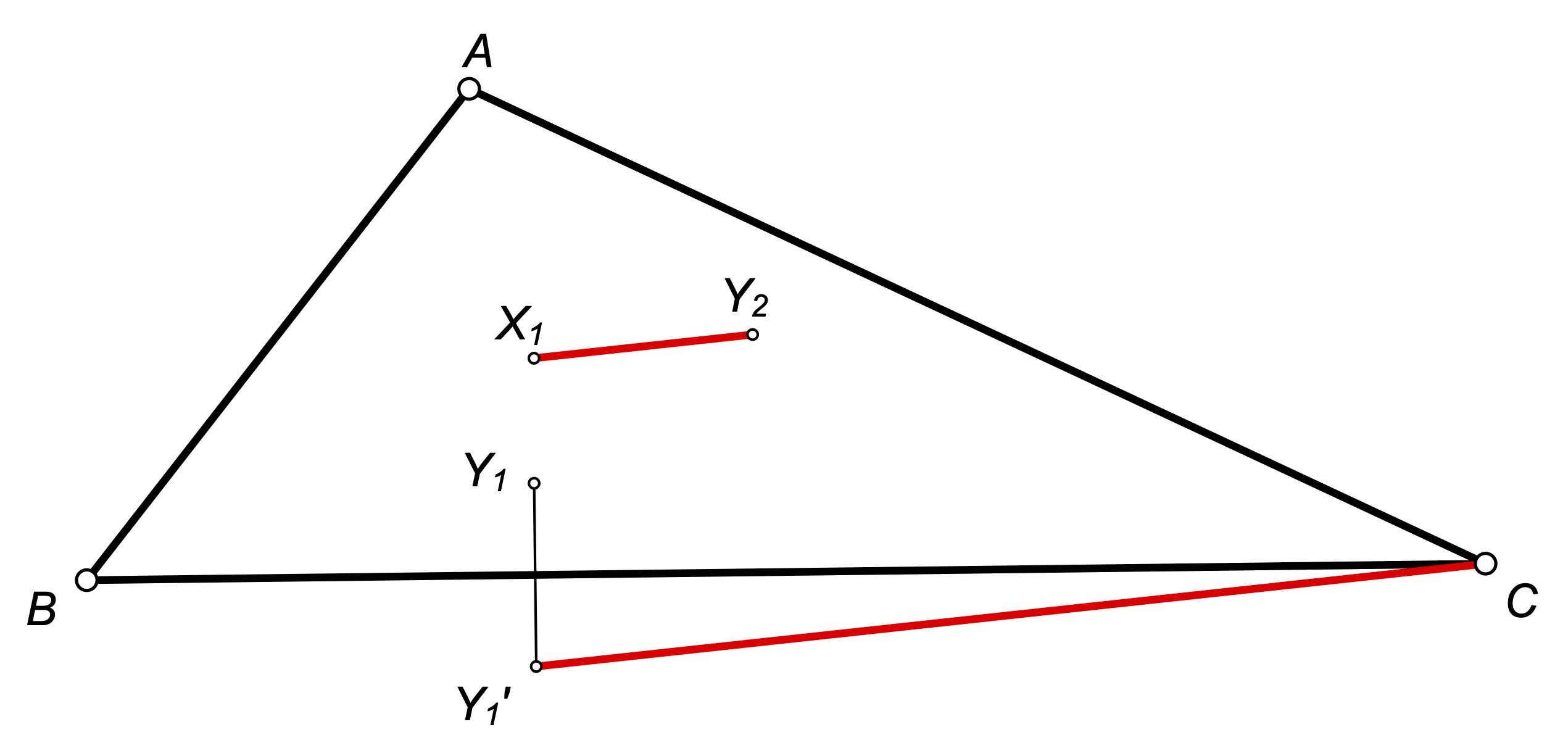}
\caption{red lines are parallel}
\label{fig:HeptaY1reflect}
\end{figure}

\begin{proof}
From Lemma~\ref{lemma:BCreflect}, we find that the coordinates of $Y_1'$ are
$$\left(a^2 u^2: a^2 (b - 2 u) u + (-b^2 + c^2) u^2 + a^3 (-b + u):u (-a^2 b + (b^2 - c^2) u)\right).$$
Using Lemma~\ref{lemma:inf}, we find that the infinite point on line $CY_1'$ is
$$i_1=\Bigl(a^2 u^2: a^2 (b - 2 u) u + (-b^2 + c^2) u^2 + a^3 (-b + u)$$
$$: a^3 (b - u) + (b^2 - c^2) u^2 + a^2 u (-b + u)\Bigr).$$
Similarly, the infinite point on the line $X_1Y_2$ is
$$i_2=\Bigl(a (b + c) (b - u) - a^2 u - (b + c) (b - u) u:u (b^2 + c u) + a (-b^2 + u^2)$$
$$: a^2 u + a (b - u) (-c + u) + u (b (c - u) - 2 c u)\Bigr)$$
The two lines will be parallel if these infinite points are the same.
The condition for $i_1$ and $i_2$ to be the same is \texttt{Cross}$[i_1,i_2]=0$.
Eliminating $u$ from this equation and Equation~(\ref{eq:u}) gives an equation
that is identically true when we let
$$a=\sin\frac{4\pi}{7},\, b =\sin\frac{2\pi}{7}, \,\mathrm{and}\,\, c=\sin\frac{\pi}{7}.$$
Thus, the lines are parallel.
\end{proof}

\begin{theorem}
\label{thm:HeptaX1}
Let $ABC$ be a heptagonal triangle (named counterclockwise) with $AB<AC<BC$
(Figure~\ref{fig:HeptaX1}).
Then $\angle CX_1Y_2+\angle X_1CY_1=\frac{\pi}{7}$.
\end{theorem}

\begin{figure}[h!t]
\centering
\includegraphics[width=0.5\linewidth]{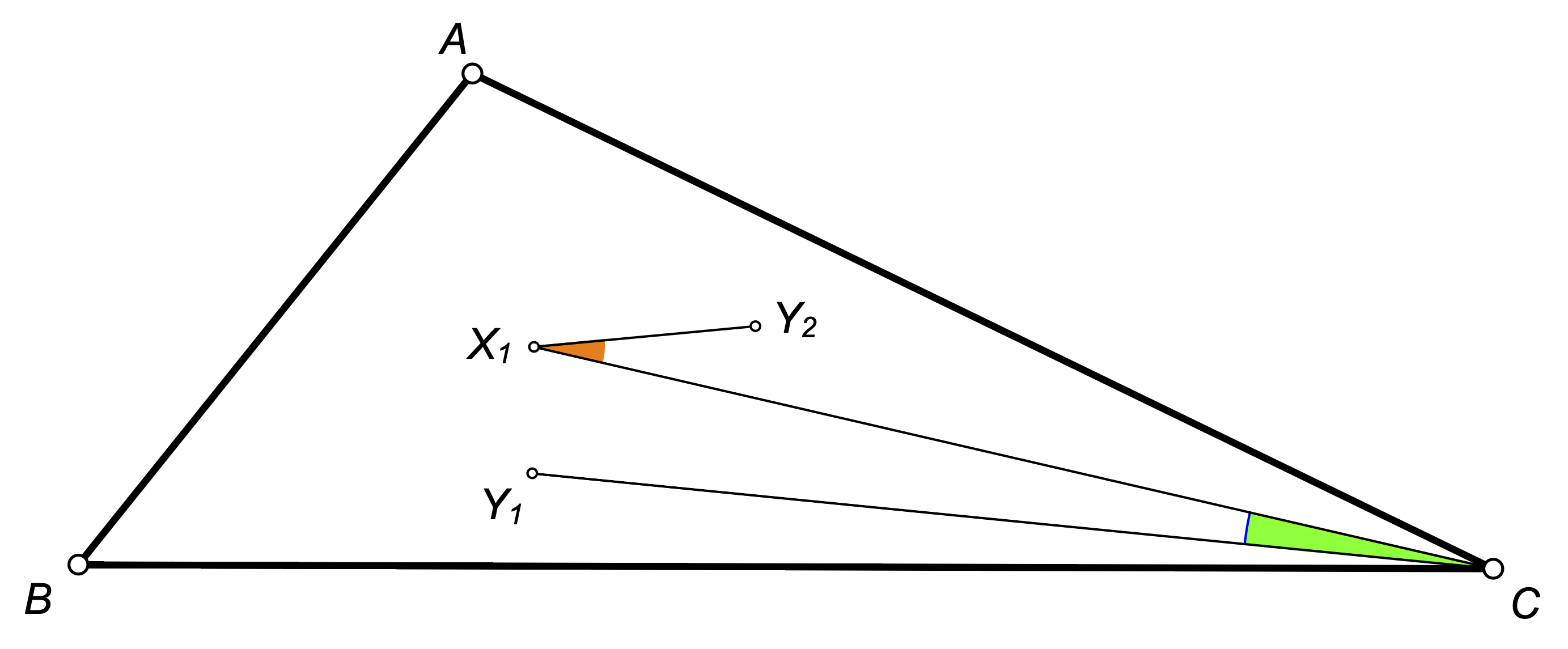}
\caption{}
\label{fig:HeptaX1}
\end{figure}

%\newpage

\begin{proof}
Let $Y_1'$ be the reflection of $Y_1$ about $BC$. Draw in line $CX_3$ as shown in Figure~\ref{fig:HeptaX1proof}.

\begin{figure}[h!t]
\centering
\includegraphics[width=0.5\linewidth]{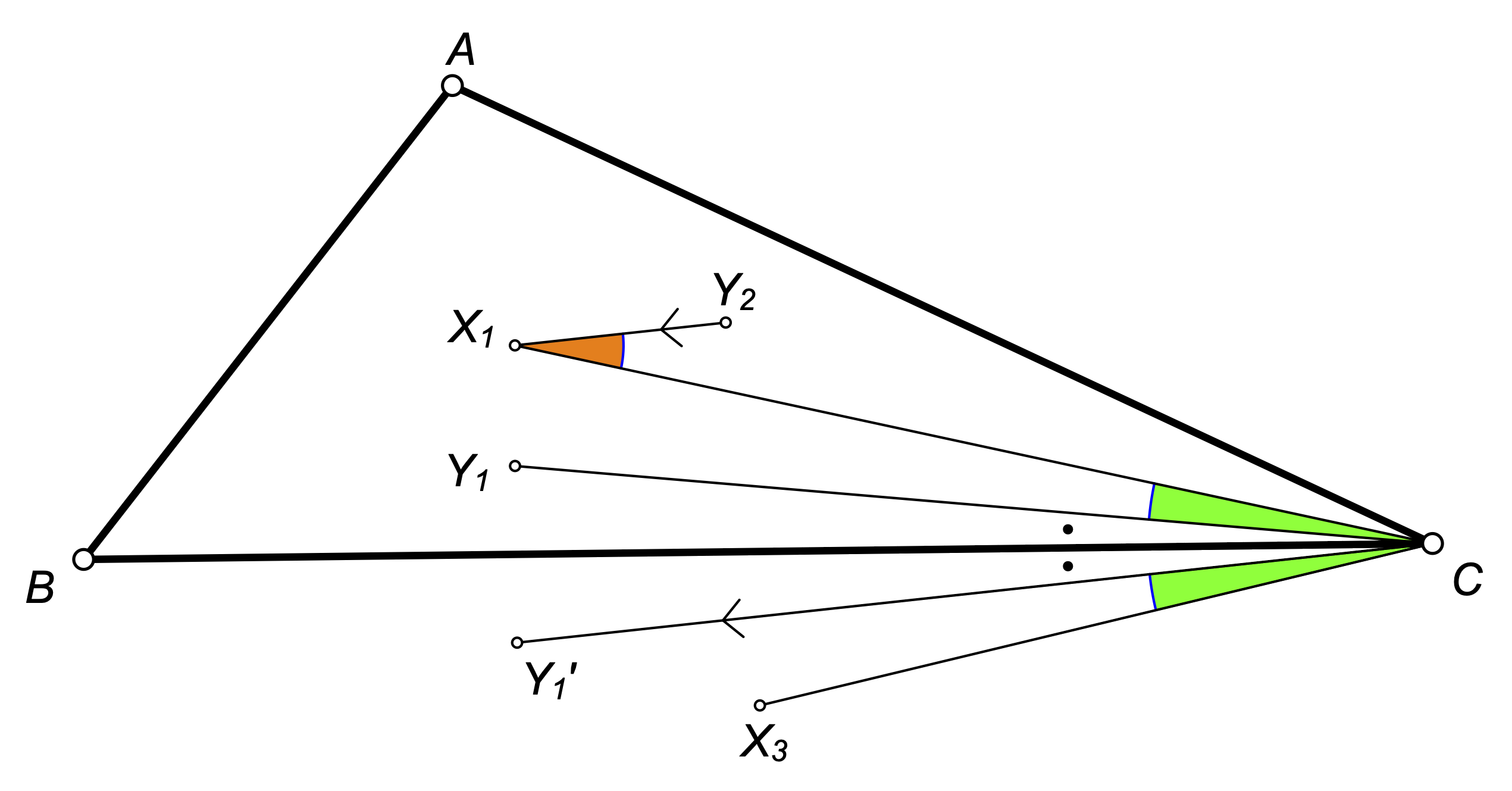}
\caption{}
\label{fig:HeptaX1proof}
\end{figure}

Since $X_3$ is the center of the circle through $A$, $B$, and $C$, $\angle CX_3B=6\pi/7$.
This implies $\angle BCX_3=\pi/14$. Since $X_1$ is the incenter, $CX_1$ bisects $\angle ACB$,
so $\angle X_1CB=\pi/14$. Therefore $\angle BCX_3=\angle X_1CB$.

Since $Y_1'$ is the reflection of $Y_1$ about side $BC$, $\angle Y_1CB=\angle BCY_1'$,
i.e. the dotted angles are equal (the ``black angles''). Subtracting, we find that $\angle X_1CY_1=\angle Y_1'CX_3$, i.e. the green
angles are equal.

From Theorem~\ref{thm:HeptaY1reflect}, we know that $Y_2X_1\parallel CY_1'$.
%Therefore $\angle Y_2X_1C=\angle X_1CY_1'=\angle X_1CB+\angle BCY_1'$.
%Hence $\angle Y_2X_1C+\angle X_1CY_1=\angle X_1CY_1'+Y_1CX_3=
%\angle X_1CB+Y_1CB+\angle BCY_1+Y_1CX_3
%=(\angle X_1CB+\angle Y_1CB)+(\angle BCY_1+Y_1CX_3)=
%=2\angle X_1CB$.
Therefore orange = green + 2 black.
Hence orange + green = 2 green + 2 black = 2(green + black) = $2\angle X_1CB=\angle ACB=\pi/7$.
\end{proof}

We found a number of results involving heptagonal triangles and center$X_1$, but they were all geometrically equivalent
to Theorem~\ref{thm:HeptaX1} because of the following result.
%which can be proven using angle chasing, Theorem~\ref{thm:HeptaX1},
%properties of $X_1$, and the known values of angles in a heptagonal triangle.

\newpage

\begin{theorem}
\label{thm:HeptaX1angs}
Let $ABC$ be a heptagonal triangle (named counterclockwise) with $AB<AC<BC$.
Let the angles be named as shown in Figure~\ref{fig:HeptaX1angs}.
If $\angle Y_2X_1A=\theta$, then

\begin{minipage}{2in}
$$
\begin{aligned}
\angle1&=\frac{2\pi}{7}\\
\angle2&=\frac{9\pi}{14}-\theta\\
\angle3&=\frac{\pi}{14}\\
\end{aligned}
$$
\end{minipage}
~
\begin{minipage}{1.5in}
$$
\begin{aligned}
\angle4&=\theta-\frac{\pi}{2}\\
\angle5&=\frac{11\pi}{14}\\
\angle6&=\frac{4\pi}{7}-\theta\\
\end{aligned}
$$
\end{minipage}
~
\begin{minipage}{2in}
$$
\begin{aligned}
\angle7&=\frac{\pi}{7}\\
\angle8&=\frac{\pi}{7}\\
\angle9&=\frac{2\pi}{7}
\end{aligned}
$$
\end{minipage}

\end{theorem}

\begin{figure}[h!t]
\centering
\includegraphics[width=0.7\linewidth]{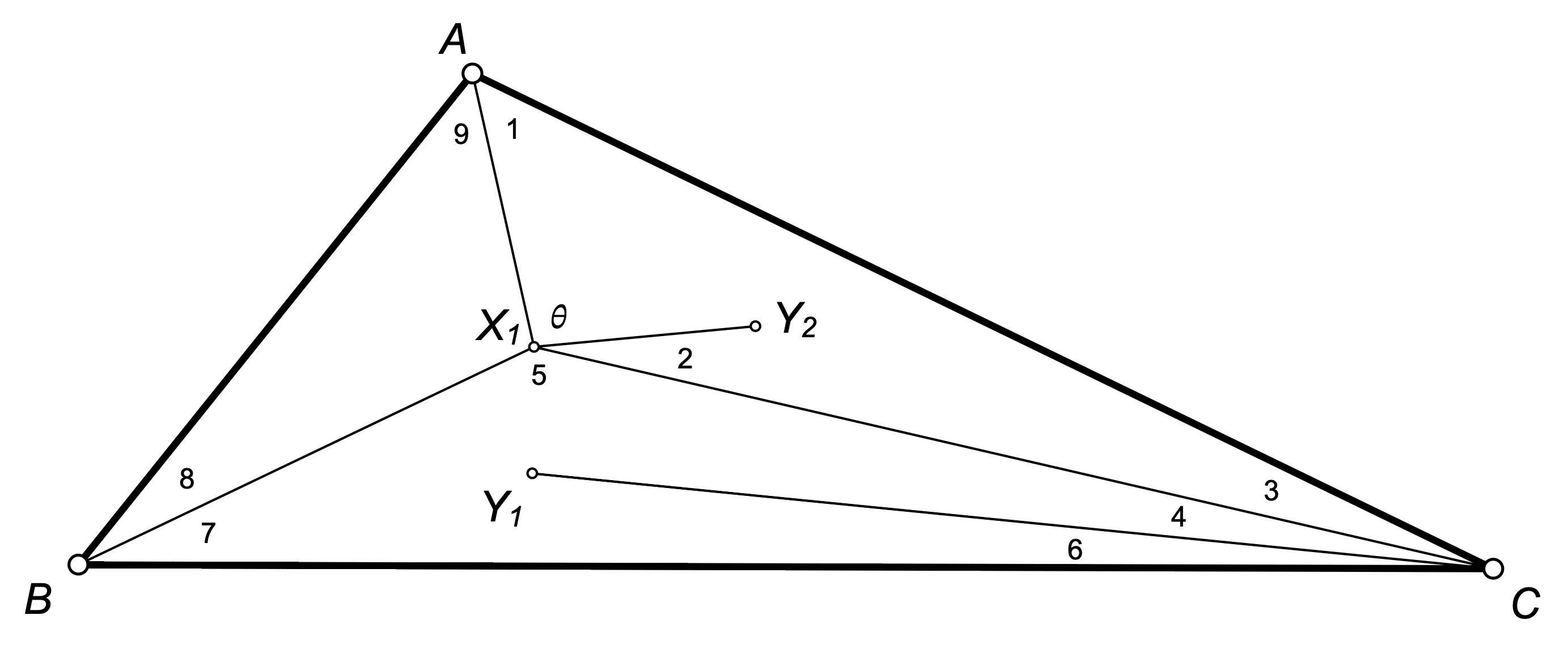}
\caption{}
\label{fig:HeptaX1angs}
\end{figure}

\begin{proof}
Since $\triangle ABC$ is a heptagonal triangle, $\angle CBA=\pi/7$, $\angle ACB=2\pi/7$, and $\angle BAC=4\pi/7$.
Since $X_1$ is the incenter of $\triangle ABC$, $AX_1$, $BX_1$, and $CX_1$ are angle bisectors.
Thus, $\angle  1=\angle 9=2\pi/7$, $\angle 3=\angle 4+\angle 6=\pi/14$, and $\angle 7=\angle 8=\pi/7$.
From $\triangle CX_1A$, we find $\angle 2=9\pi/14-\theta$.
From $\triangle BX_1C$, we find $\angle 5=11\pi/14$.
From Theorem~\ref{thm:HeptaX1}, $\angle 2+\angle 4=\pi/7$.
Therefore $\angle 4=\pi/7-(9\pi/14-\theta)=\theta-\pi/2$.
Finally, $\angle 6=\angle X_1CB-\angle 4=\pi/14-(\theta-\pi/2)=4\pi/7-\theta$.
\end{proof}

%\newpage

\begin{theorem}
Let $ABC$ be a heptagonal triangle (named counterclockwise) with $AB<AC<BC$
(Figure~\ref{fig:HeptaX1X3}).
Then $\displaystyle\angle CX_1Y_2=\angle Y_1CX_3$.
\end{theorem}

\begin{figure}[h!t]
\centering
\includegraphics[width=0.6\linewidth]{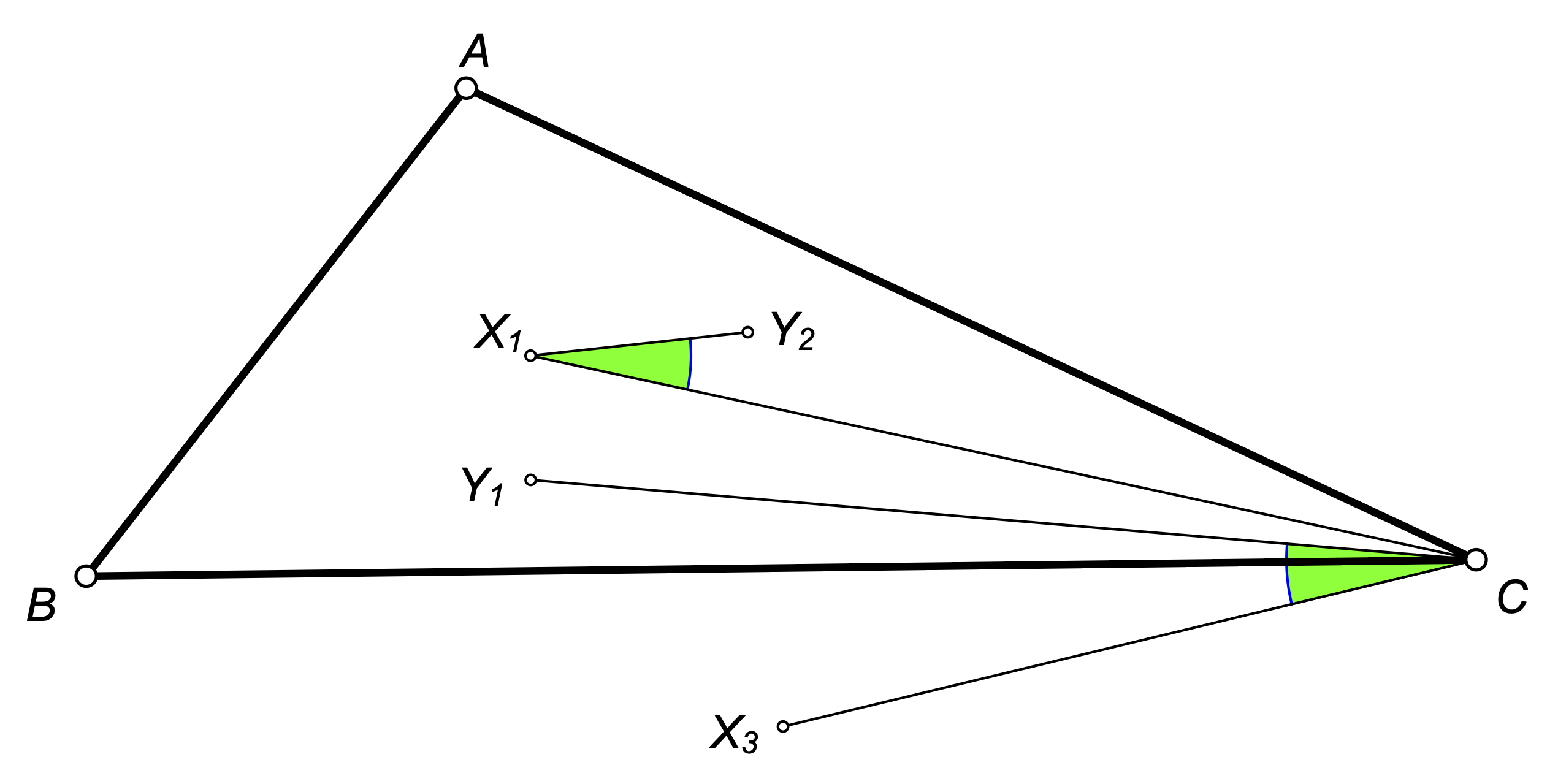}
\caption{green angles are equal}
\label{fig:HeptaX1X3}
\end{figure}

\begin{proof}
Since $X_3$ is the center of the circle through $A$, $B$, and $C$, $\angle CX_3B=6\pi/7$.
This implies $\angle BCX_3=\pi/14$. Then, using Theorem~\ref{thm:HeptaX1angs}, we have
$\angle Y_1CX_3=\angle BCX_3+\angle 6=\pi/14+4\pi/7-\theta=9\pi/14-\theta=\angle 2$.
\end{proof}

\newpage

We also found the following three results, involving other centers.
The proofs, using Mathematica, involved a lot of computation, so we have omitted the details.

\begin{theorem}
Let $ABC$ be a heptagonal triangle (named counterclockwise) with $AB<AC<BC$.
Let $P$ be the foot of the perpendicular from $Y_1$ to $AC$ (Figure~\ref{fig:HeptaX5}).
Then $\angle CBX_5=\angle Y_2Y_1P$.
\end{theorem}

\begin{figure}[h!t]
\centering
\includegraphics[width=0.5\linewidth]{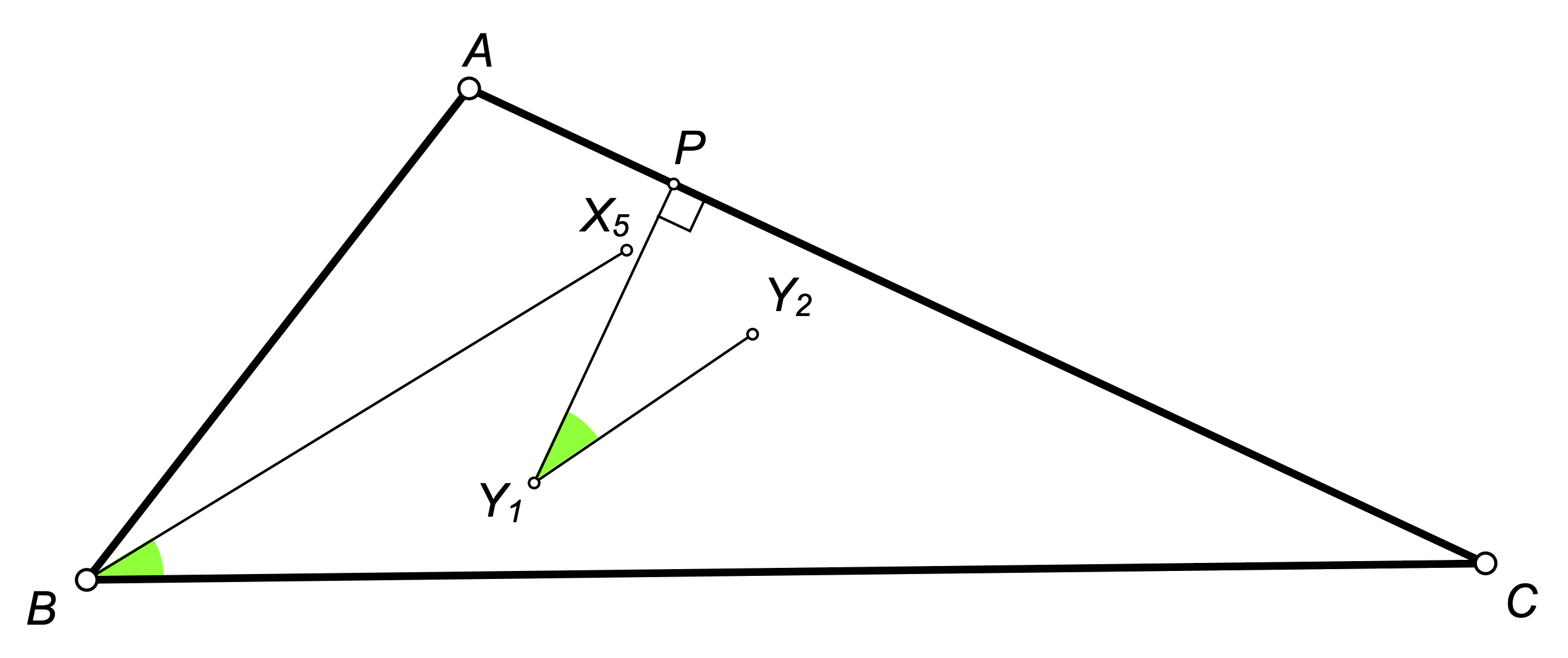}
\caption{green angles are equal}
\label{fig:HeptaX5}
\end{figure}

\void{
\begin{proof}
\redtext{to be supplied}
\end{proof}
}

\begin{theorem}
Let $ABC$ be a heptagonal triangle (named counterclockwise) with $AB<AC<BC$
(Figure~\ref{fig:HeptaX1X21}).
Then $\angle X_{21}X_1Y_2=\angle CY_1Y_2$.
\end{theorem}

\begin{figure}[h!t]
\centering
\includegraphics[width=0.6\linewidth]{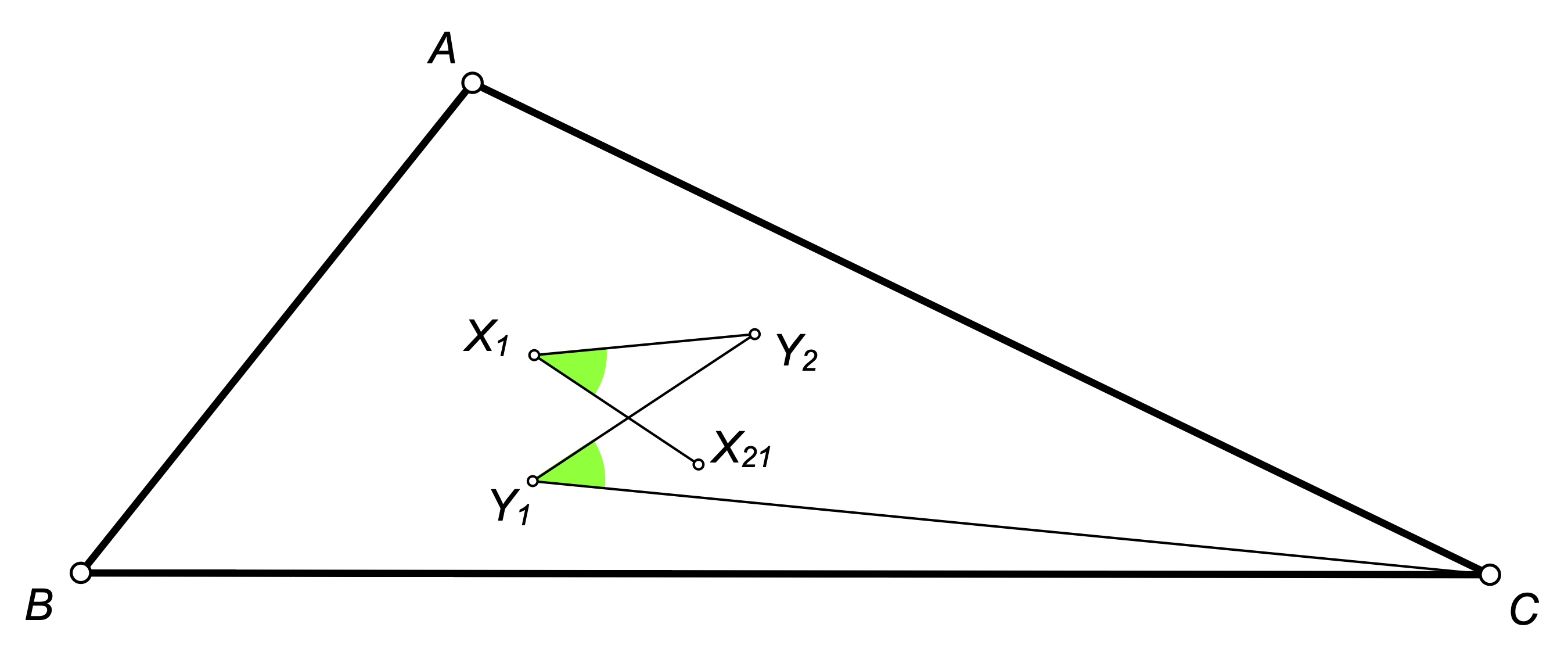}
\caption{green angles are equal}
\label{fig:HeptaX1X21}
\end{figure}

\void{
\begin{proof}
\redtext{to be supplied}
\end{proof}
}

\begin{theorem}
Let $ABC$ be a heptagonal triangle (named counterclockwise) with $AB<AC<BC$
(Figure~\ref{fig:HeptaX1X28}).
Then $\angle X_1Y_2Y_1=\angle X_{28}CY_1$.
\end{theorem}

\begin{figure}[h!t]
\centering
\includegraphics[width=0.6\linewidth]{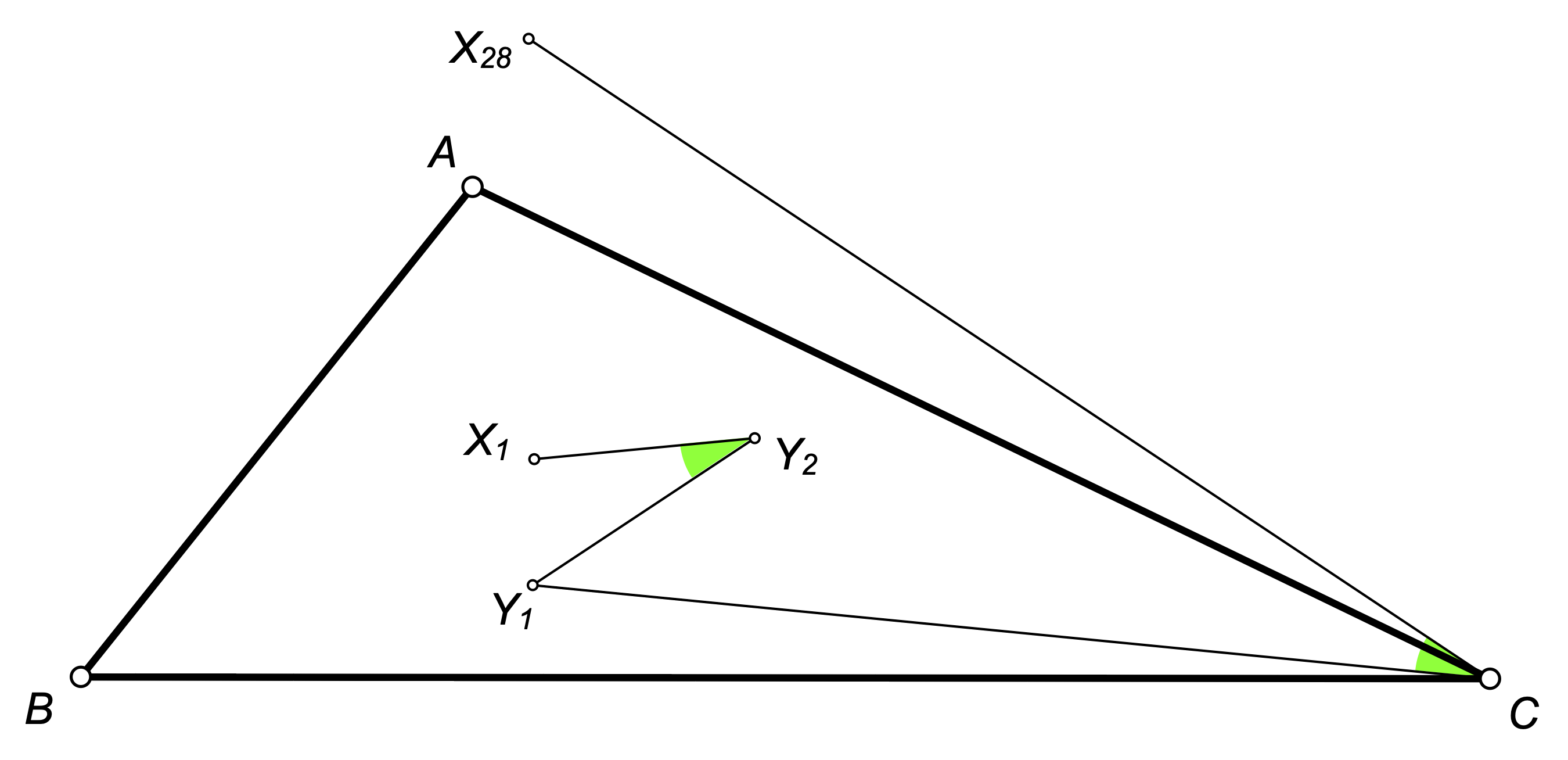}
\caption{green angles are equal}
\label{fig:HeptaX1X28}
\end{figure}

\void{
\begin{proof}
\redtext{to be supplied}
\end{proof}
}

\begin{open}
Is there a purely synthetic geometric proof of any of the three preceding theorems, or a proof
that does not involve massive amounts of computation?
\end{open}

\newpage

\section{Harmonic Triangles}

\begin{theorem}
\label{thm:medial}
Let $ABC$ be a triangle such that $Y_1Y_2$ passes through a vertex.
Then $Y_1Y_2$ is a median of the triangle.
\end{theorem}

\begin{proof}
Without loss of generality, assume that $Y_1Y_2$ passes through vertex B 
and $\triangle ABC$ is named counterclockwise (Figure~\ref{fig:harmonic}).

\begin{figure}[h!t]
\centering
\includegraphics[width=0.6\linewidth]{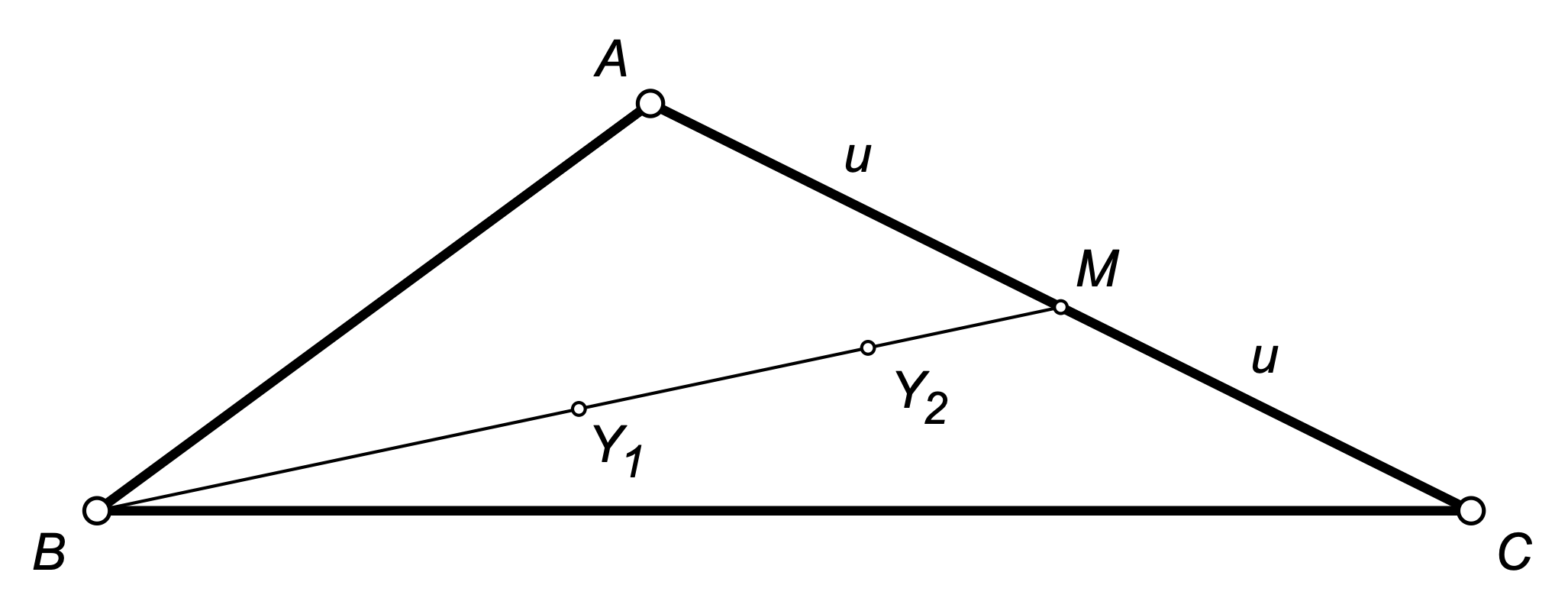}
\caption{$B$, $Y_1$, and $Y_2$ are collinear}
\label{fig:harmonic}
\end{figure}

Let $BY_1$ meet $AC$ at $M$. Then, from Figure~\ref{fig:YffPoints} (left), we have $CM=u$.
Line $BY_2$ is the same line, so also meets $AC$ at $M$.
From Figure~\ref{fig:YffPoints} (right), we have $AM=u$.
Thus, $AM=CM$ and $BM$ is a median of the triangle.
\end{proof}

\begin{theorem}
\label{thm:median}
Let $ABC$ be a triangle such that $BY_1$ is a median.
Then $Y_1$ divides the median into two segments in the ratio $2b:2a-b$.
\end{theorem}

\begin{proof}
Let $BY_1$ meet $AC$ at $E$ and let $AY_1$ meet $BC$ at $D$. % (Figure~\ref{fig:harmonicProof}).
\begin{figure}[h!t]
\centering
\includegraphics[width=0.5\linewidth]{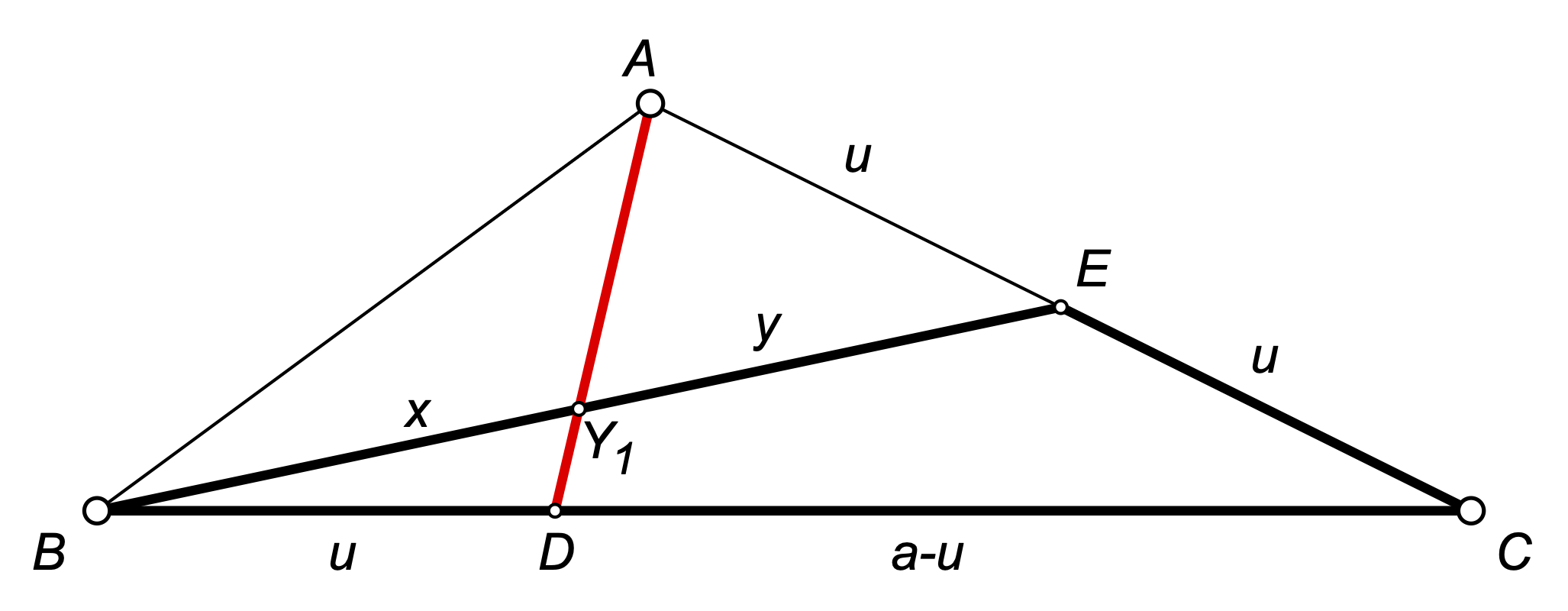}
%\caption{}
\label{fig:harmonicProof}
\end{figure}

Then, from Figure~\ref{fig:YffPoints} we see that $CE=BD=u$.
Since $BY_1$ is a median, $AE=CE$, so $AE=u$.
Since $BD=u$, we have $CD=a-u$.
Line $AD$ is a transversal of $\triangle EBC$, so by the Theorem of Menelaus, we have
$x(a-u)u=yu(2u)$, so $x/y=2u/(a-u)$. But $u=b/2$. Thus $x/y=b/(a-\frac{b}2)$ which implies
$x/y=2b/(2a-b)$.
\end{proof}

\begin{corollary}
If $Y_1$ is the midpoint of a median, then the triangle is similar to a 3--4--6 triangle.
\end{corollary}

\begin{theorem}
Let $ABC$ be a triangle such that $BY_2$ is a median.
Then $Y_2$ divides the median into two segments in the ratio $2(2a-b):b$.
\end{theorem}

\begin{proof}
The proof is similar to the proof of Theorem~\ref{thm:median}, so the details are omitted.
\end{proof}

\newpage

The \emph{harmonic mean} of $x$ and $y$ is $\frac{2}{1/x+1/y}$.
A triangle is called a \emph{harmonic triangle} if one side is the harmonic mean of the other two sides.
We say that the triangle is \emph{harmonic}.

\begin{theorem}
\label{thm:vertex}
A triangle is harmonic if and only if $Y_1Y_2$ passes through a vertex.
\end{theorem}

\begin{proof}
If $Y_1Y_2$ passes through a vertex, assume
without loss of generality, that the vertex is $B$
and $\triangle ABC$ is named counterclockwise (Figure~\ref{fig:harmonic}).
From Theorem~\ref{thm:medial}, $BY_1$ is a median and so $u=b/2$.
%From Theorem~\ref{thm:yff}, we have $u^3=(a-u)(b-u)(c-u)$.
Letting $u=b/2$ in Equation~(\ref{eq:u}) gives $b(ab-2ac+bc)=0$ which implies $b=2ac/(a+c)$
and the triangle is harmonic.

Conversely, if the triangle is harmonic, assume without loss of generality that $2/b=1/a+1/c$.
The barycentric coordinates for $B$ are $(0:1:0)$.
Using the simple coordinates for $Y_1$ and $Y_2$,
we find that that $Y_1Y_2$ passes through $B$ if and only if
$$\left|
\begin{array}{ccc}
u^2&(a-u)(b-u)&(b-u)u\\
(a-u)(b-u)&u_2&(a-u)u\\
0&1&0
\end{array}
\right|=0.$$
Expanding this gives $b (b - 2 u) (a - u) u=0$.
From  \cite{Yff}, it is known that $0<u<a$, so therefore,
$Y_1Y_2$ passes through vertex $B$ in a harmonic triangle, if and only if $b=2u$.\\
If the triangle is harmonic, then $b=2u$, so $b-u=u$, and from Figure~\ref{fig:YffPoints},
$BY_1$ is a median. Similarly, $BY_2$ is a median.
Therefore $Y_1Y_2$ passes through $B$.
\end{proof}

\begin{theorem}
\label{thm:X2}
A triangle is harmonic if and only if $Y_1Y_2$ passes through $X_2$.
\end{theorem}

\begin{proof}
Assume that $Y_1Y_2$ passes through $X_2$.
The barycentric coordinates for $X_2$ are  $(1:1:1)$. Using the simple coordinates for $Y_1$ and $Y_2$,
the determinant condition for $Y_1$, $Y_2$ and $X_2$ to be collinear is
$$\left|
\begin{array}{ccc}
u^2&(a-u)(b-u)&(b-u)u\\
(a-u)(b-u)&u_2&(a-u)u\\
1&1&1
\end{array}
\right|=0.$$
Expanding this out gives $(a - 2 u) (b - 2 u) (b u -ab + au)=0$. Eliminating variable $u$ between this equation and 
Equation~(\ref{eq:u}) gives
$$a^3 b^3 (a b+a c-2 b c) (2 a b-a c-b c) (a b-2 a c+b c)=0$$
which shows that $\triangle ABC$ is harmonic.

Conversely, if the triangle is harmonic, then by Theorem~\ref{thm:vertex}, $Y_1Y_2$
passes through a vertex. From Theorem~\ref{thm:medial}, we can conclude that $Y_1Y_2$ is a median.
Thus, $Y_1Y_2$ passes through $X_2$.
\end{proof}

\begin{theorem}
A triangle is harmonic if and only if $Y_1Y_2$ bisects a side of the triangle.
\end{theorem}

\begin{proof}
Without loss of generality.
assume that $Y_1Y_2$ passes through $E$, the midpoint of $AC$.
The barycentric coordinates for $E$ are  $(1:0:1)$.
Using the simple coordinates for $Y_1$ and $Y_2$,
the determinant condition for $Y_1$, $Y_2$ and $E$ to be collinear is
$$\left|
\begin{array}{ccc}
u^2&(a-u)(b-u)&(b-u)u\\
(a-u)(b-u)&u_2&(a-u)u\\
1&0&1
\end{array}
\right|=0.$$
Expanding this out gives $(b-2 u)(a^2 (b-u)+2 a u (u-b)+b u^2)=0$.
Eliminating variable $u$ between this equation and 
Equation~(\ref{eq:u}) gives
\begin{equation}
a^3 b^4 (a b - 2 a c + b c) (a^2 b + 2 a^2 c - 2 a b c + 2 a c^2 + b c^2) = 0.
\end{equation}
If $a^2 b + 2 a^2 c - 2 a b c + 2 a c^2 + b c^2 = 0$, then
$b= -2ac (a+c)/(a-c)^2$.
This equation is impossible since the left side is positive and the right side is negative.
Thus $a b - 2 a c + b c=0$ which implies that  the triangle is harmonic.

Conversely, if the triangle is harmonic, then by Theorem~\ref{thm:vertex}, $Y_1Y_2$
passes through a vertex. From Theorem~\ref{thm:medial}, we can conclude that $Y_1Y_2$ is a median.
Thus, $Y_1Y_2$ passes through $E$.
\end{proof}

\section{AP Triangles}

An \emph{AP Triangle} is a triangle whose sides are in arithmetic progression.

\begin{theorem}
\label{thm:AP}
Let $\triangle ABC$ (named counterclockwise), be an AP Triangle, with $2a=b+c$.
Then $AX_1\parallel Y_1Y_2$.
\end{theorem}

\begin{figure}[h!t]
\centering
\includegraphics[width=0.4\linewidth]{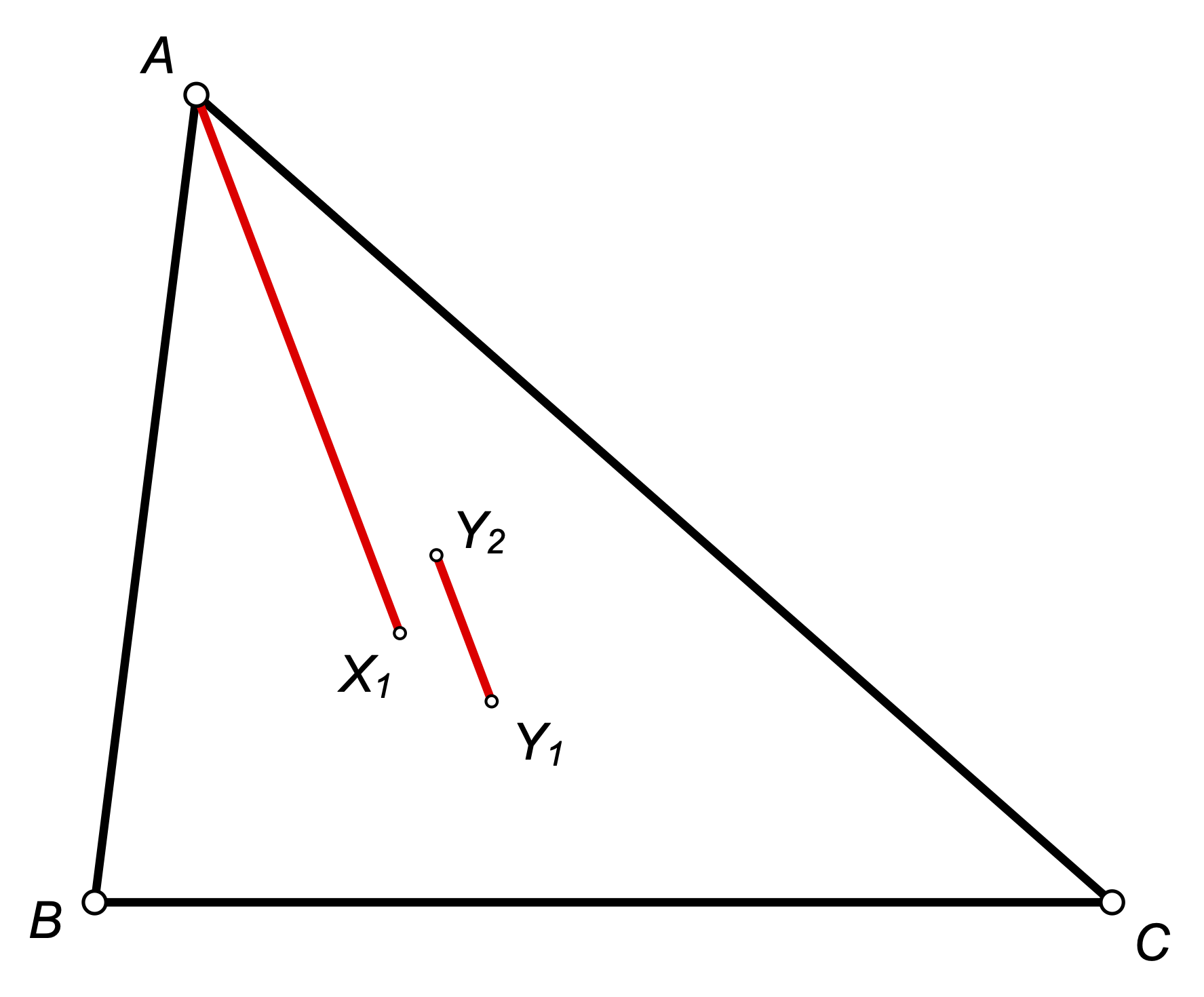}
\caption{red lines are parallel}
\label{fig:APTriangle}
\end{figure}

\begin{proof}
According to \cite{BicentricPairs}, the point at infinity on $Y_1Y_2$ is $X_{513}$.
From \cite{ETC}, we find that the barycentric coordinates of $X_{513}$ are
$$I_1=\left(a(b-c):b(c-a):c(a-b)\right).$$
Since $A=(1:0:0)$ and $X_1=(a:b:c)$, the point at infinity on $AX_1$ is
$$I_2=(-b-c:b:c)$$
When $b=2a-c$, we find that $I_1=(-2a:2a-c:c)$ and $I_2=(-2a:2a-c:c)$.
Thus, the two lines have the same point at infinity.
Therefore $AX_1\parallel Y_1Y_2$.
\end{proof}

\section{Double-Angle Triangles}

A \emph{double-angle triangle} is a triangle in which one angle is twice another angle.
Note that all heptagonal triangles are double-angle triangles.

The following result is fairly well known.

\newpage

\begin{lemma}[Double Angle Condition]
\label{lemma:doubleAngle}
Let $ABC$ be a triangle in which $\angle A=2\angle B$.
Then $a^2=b(b+c)$.
\end{lemma}

\begin{proof}
Let $\angle CBA=\theta$, so that $\angle BAC=2\theta$.
Extend $CA$ past $A$ to a point $D$ such that $AD=AB$,
so that $ABD$ is an isosceles triangle and $\angle ABD=\angle BDA$.

\begin{figure}[h!t]
\centering
\includegraphics[width=0.4\linewidth]{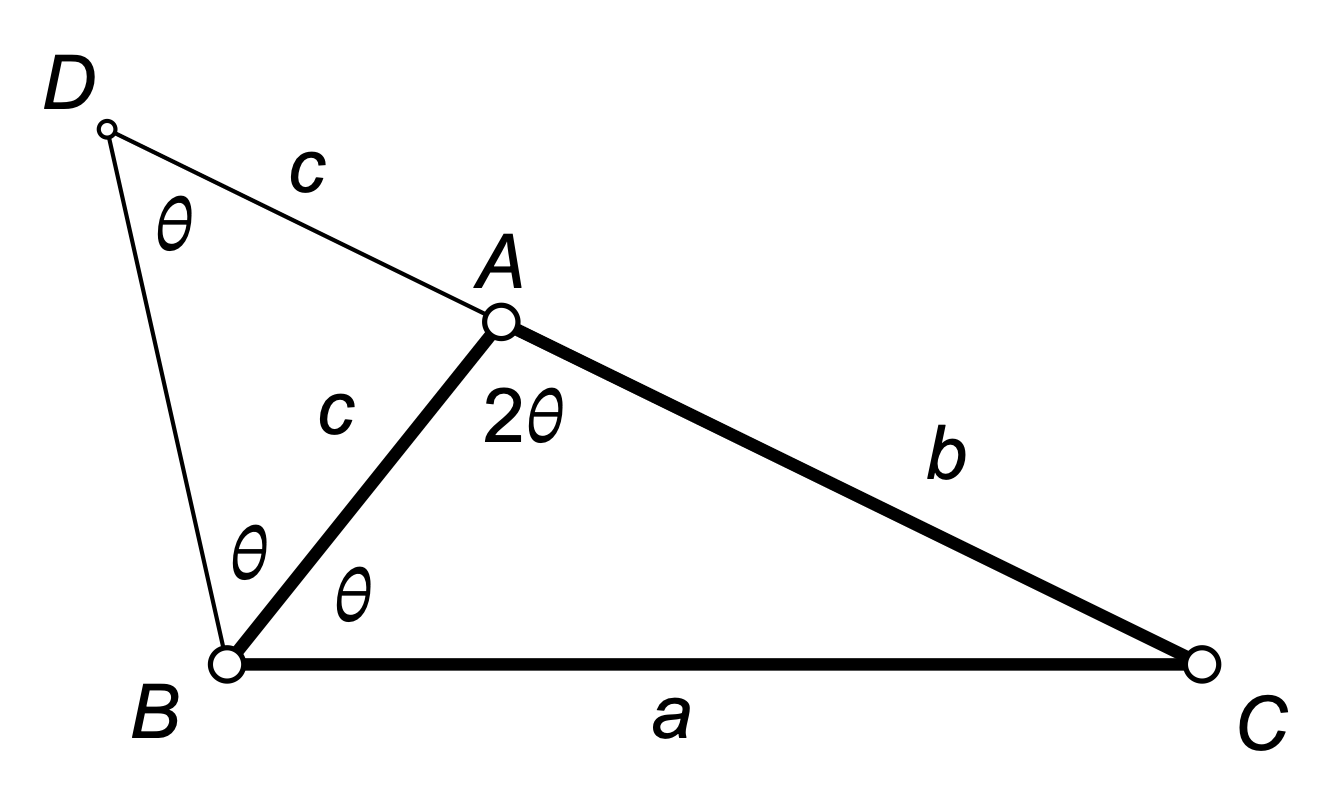}
%\caption{red lines are parallel}
\label{fig:DoubleAngleLemma}
\end{figure}

By the Exterior Angle Theorem, $\angle BAC=\angle ABD+\angle BDA$.
Thus, $\angle ABD=\angle BDA=\theta$ and $\triangle BAC\sim \triangle DBC$.
Hence $AC/BC=BC/DC$ or $b/a=a/(b+c)$.

\end{proof}

\begin{theorem}
\label{thm:doubleAngle}
Let $ABC$ be a double-angle triangle (named counterclockwise) with $\angle A=2\angle B$
(Figure~\ref{fig:doubleAngle}).
Let $Y_2'$ be the reflection of $Y_2$ about side $BC$.
Then $BY_2$, $CX_2$, and $AY_2'$ are concurrent.
\end{theorem}

\begin{figure}[h!t]
\centering
\includegraphics[width=0.6\linewidth]{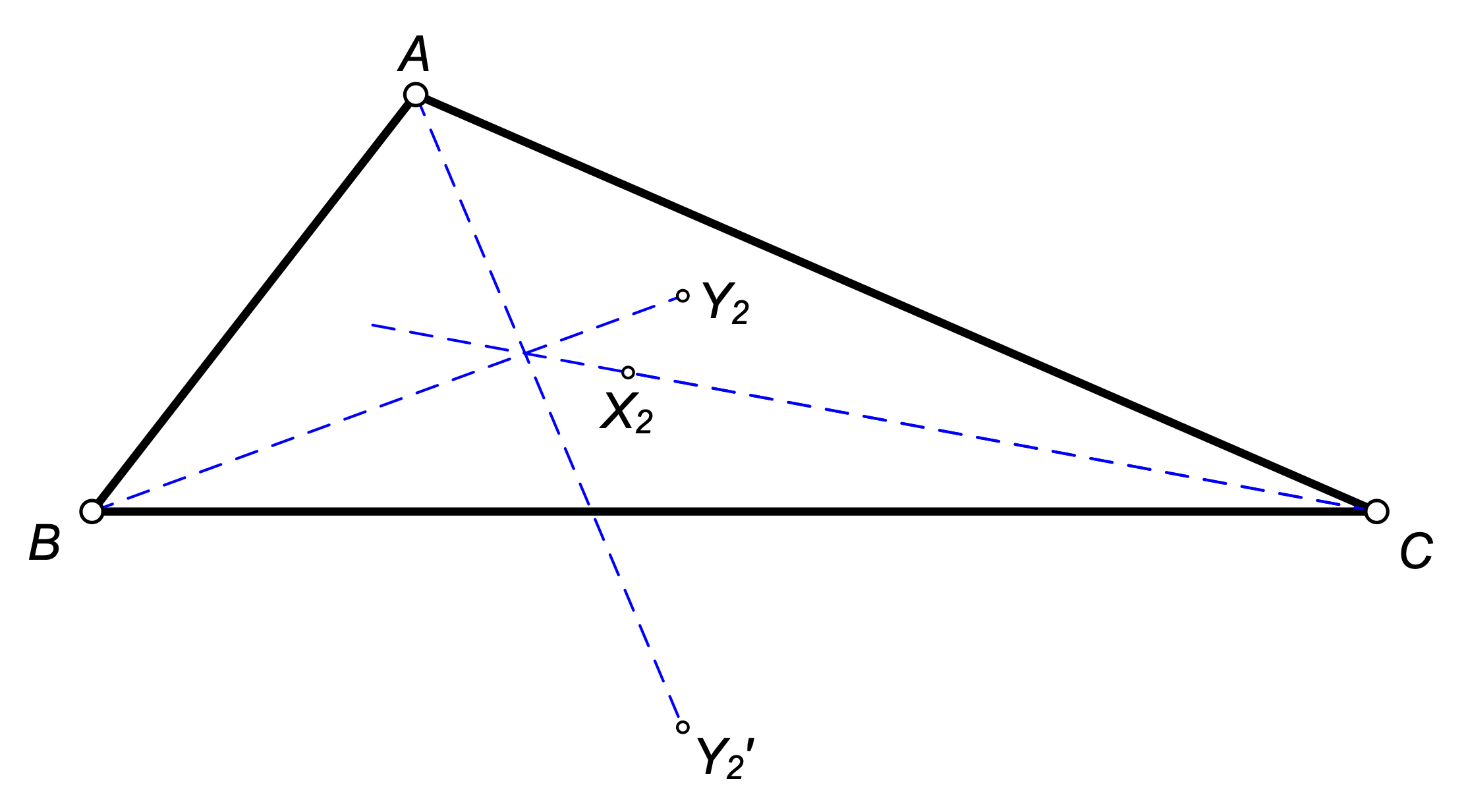}
\caption{dashed lines are concurrent}
\label{fig:doubleAngle}
\end{figure}

\begin{proof}
From Lemma~\ref{lemma:BCreflect}, we find that the coordinates of $Y_2'$ are
$$\Bigl(a^2 (a - u) (b - u): (-b^2 + c^2) (a - u) (b - u) + a^2 ((b - 2 u) u + a (-b + u))$$
$$: (-a^2 b + (b^2 - c^2) (b - u)) (a - u)\Bigr).$$
We can therefore write down the equations for the three lines $AY_2'$, $BY_2$, and $CX_2$.
The determinant condition for these three lines to be concurrent simplifies to
$$(a - u) (-a b (b^2 - c^2) (b - u) + a^3 (b - u)^2 + 
   b (b^2 - c^2) (b - u) u - a^2 u (b^2 - 2 b u + 2 u^2))=0.$$
Eliminating variable $u$ between this equation and Equation~(\ref{eq:u}) gives
$$a b (b - c) (a^2 - b (b + c)) = 0.$$
Thus, the three lines are concurrent if $a^2=b(b+c)$.
The theorem now follows from Lemma~\ref{lemma:doubleAngle}.
\end{proof}

\void{
\begin{open}
Is there an analog of Theorem~\ref{thm:doubleAngle} involving the point $Y_1$ instead of $Y_2$?
\end{open}
}

\section{Isosceles Triangles}

\begin{theorem}
Let $Y$ be one of the Yff points of $\triangle ABC$.
Let the areas of triangles $BCY$, $CAY$, and $ABY$ be $K_a$, $K_b$, and $K_c$,
respectively. Then $K_bK_c=K_a^2$ if and only if $\triangle ABC$ is isosceles with $AB=AC$.
\end{theorem}

\begin{figure}[h!t]
\centering
\includegraphics[width=0.4\linewidth]{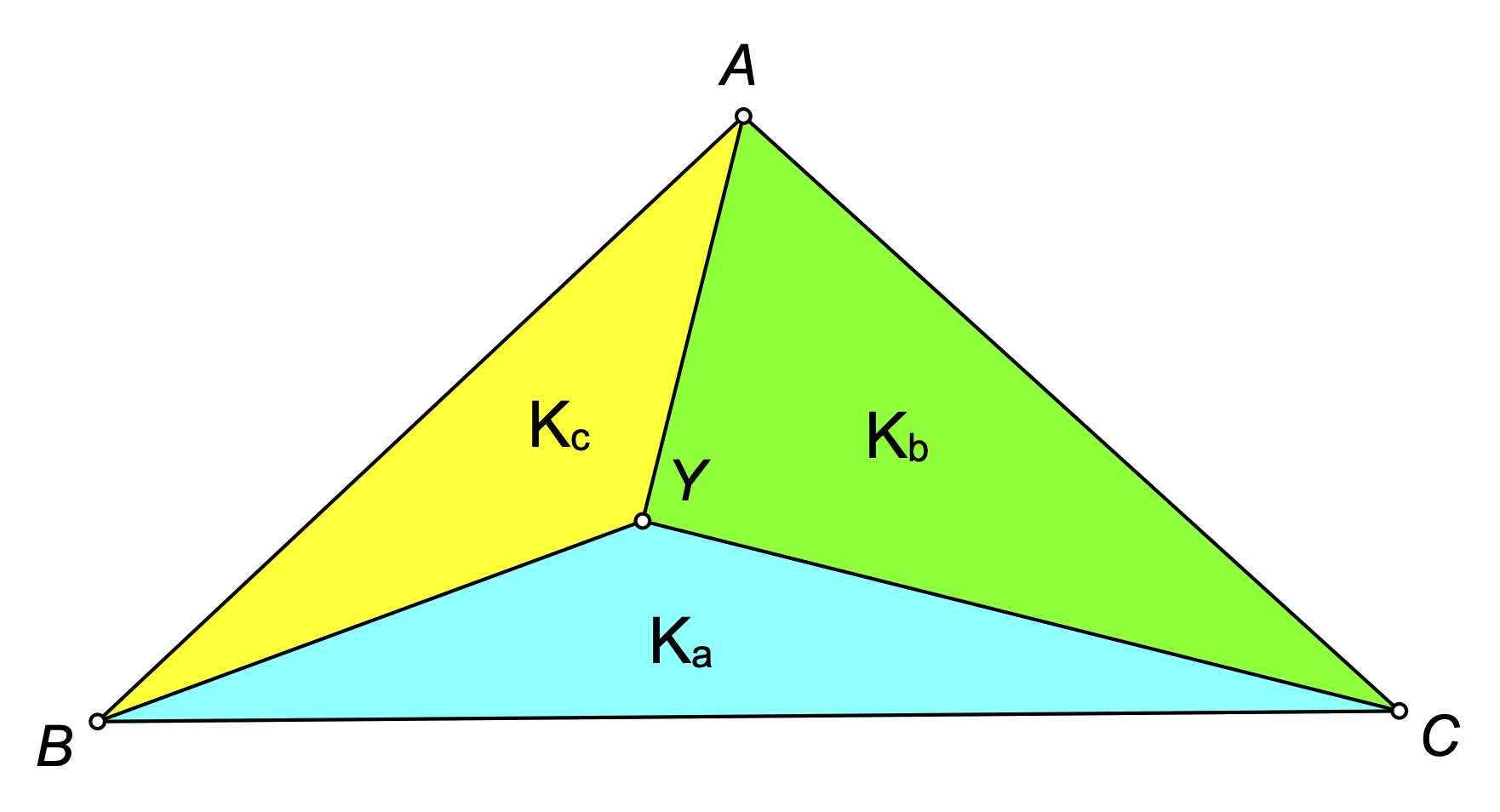}
%\caption{dashed lines are concurrent}
\label{fig:isosceles}
\end{figure}

\begin{proof}
Start with $Y=Y_1$.
The areas of these triangles are proportional to the barycentric coordinates of $Y_1$.
Then from Equation~(\ref{Y1Y2}), without loss of generality, we have
$K_a=u^2$, $K_b=(a-u)(b-u)$, and $K_c=(b-u)u$.
Thus, $K_bK_c=K_a^2$ is equivalent to
\begin{equation}
\label{eq:isos}
(a - u) u (-b + u)^2 = u^4.
\end{equation}
Eliminating $u$ from Equations (\ref{eq:isos}) and (\ref{eq:u}), this equation becomes
$a^4 b^4 (b - c) c = 0$.
This condition is true if and only if $b=c$.
The proof when $Y=Y_2$ is similar.
\end{proof}

\section{Isosceles Right Triangles}

The following two results were found by computer.
We omit the complicated proofs.

\begin{theorem}
Let $ABC$ be an isosceles right triangle (named counterclockwise) with right angle at $A$.
Let $X_1Y_2$ meet $BC$ at $D$ (Figure~\ref{fig:isoscelesRight}).
Then $\angle DY_2C=\angle CBX_8$.
\end{theorem}

\begin{figure}[h!t]
\centering
\includegraphics[width=0.4\linewidth]{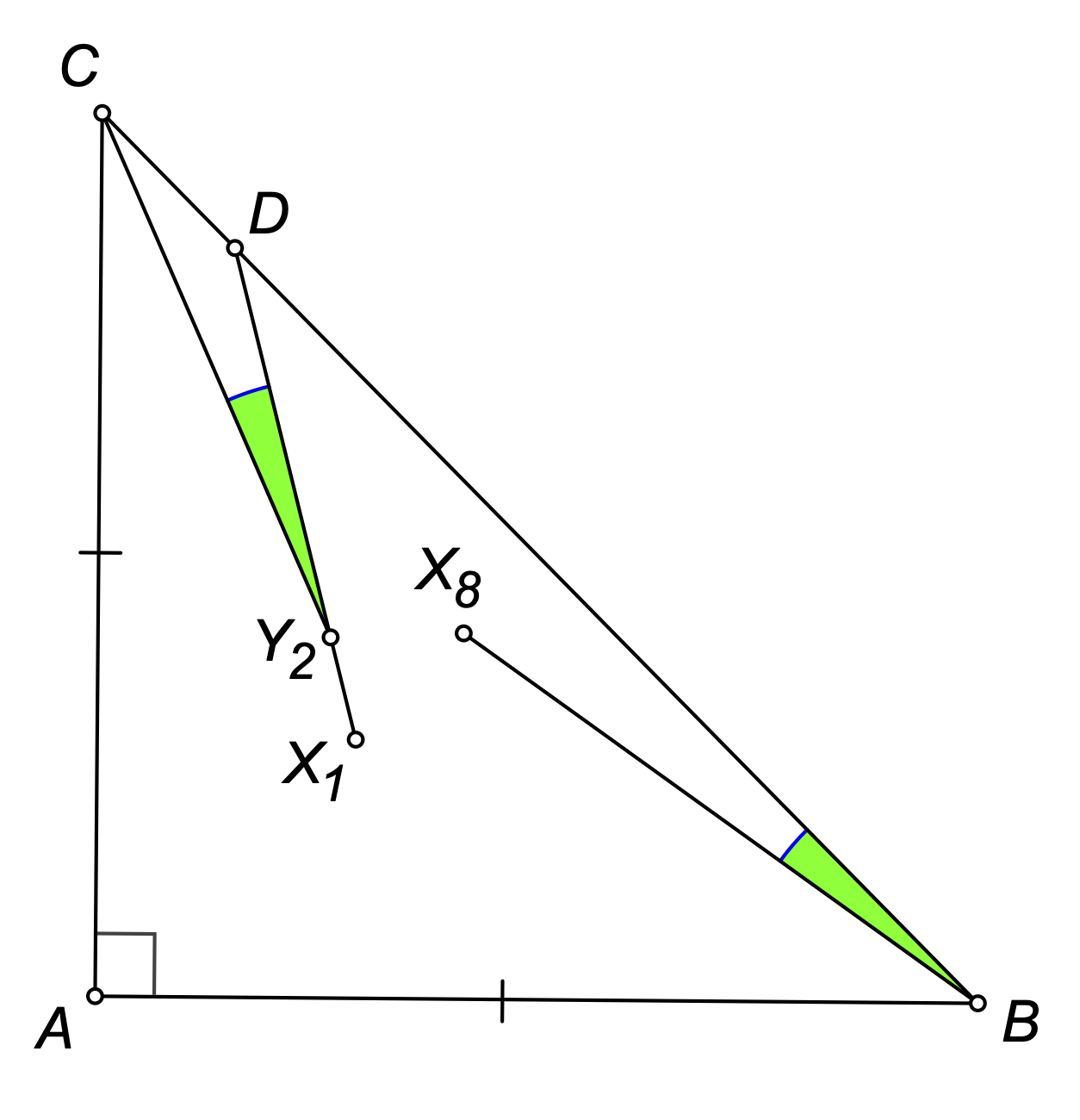}
\caption{green angles are equal}
\label{fig:isoscelesRight}
\end{figure}

\newpage

\begin{theorem}
Let $ABC$ be an isosceles right triangle (named counterclockwise) with right angle at $A$.
Then $\angle Y_2X_1C=\angle Y_2CX_9$ (Figure~\ref{fig:Mittenpunkt}).
\end{theorem}

\begin{figure}[h!t]
\centering
\includegraphics[width=0.3\linewidth]{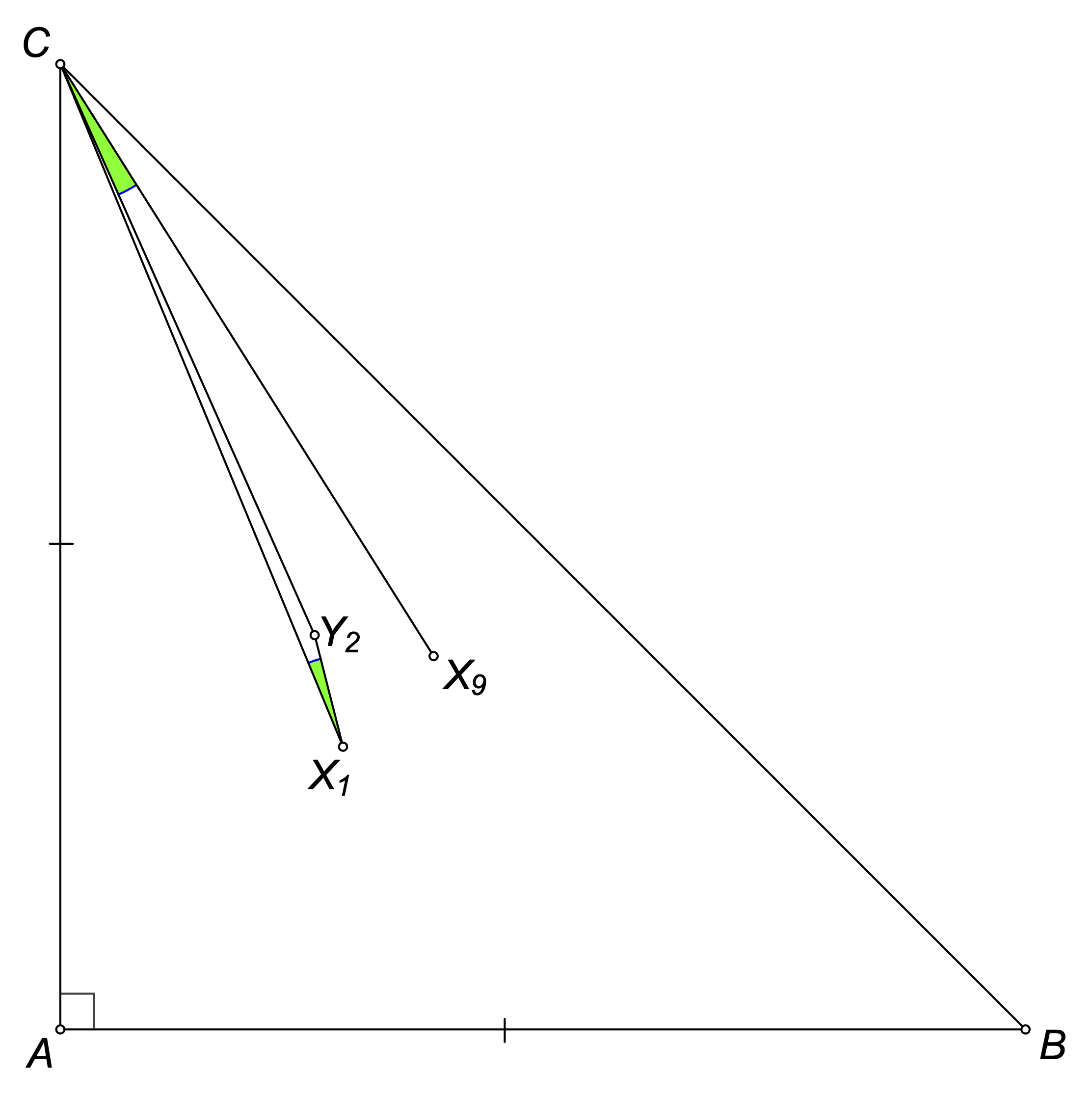}
\caption{green angles are equal}
\label{fig:Mittenpunkt}
\end{figure}

\section{Special Triangles}

\begin{theorem}
\label{thm:Nagel}
In $\triangle ABC$ (named counterclockwise), $BY_1$ passes through $X_8$
if and only if $$a^2(a-b-c)=bc(b+c-3a).$$
\end{theorem}

Figure~\ref{fig:B-Y1-X8} shows an example.

\begin{figure}[h!t]
\centering
\includegraphics[width=0.4\linewidth]{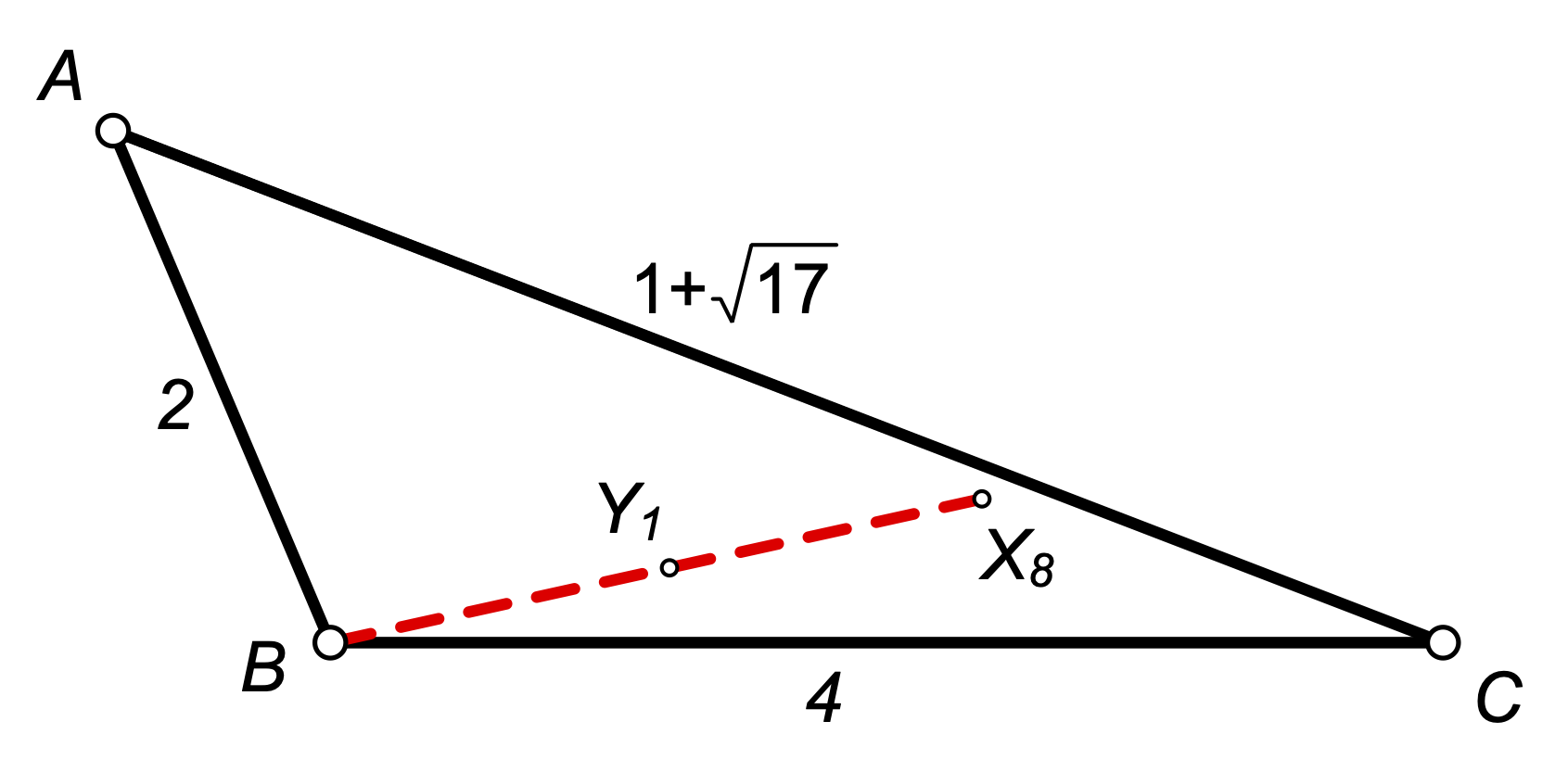}
\caption{$B$, $Y_1$, and $X_8$ are collinear}
\label{fig:B-Y1-X8}
\end{figure}

\begin{proof}
Let $BX_8$ meet $AC$ at $E$.
Then by a well-known property of the Nagel point, $CE=s-a$
where $s$ is the semiperimter of the triangle \cite[\S 3.2.2]{Yiu}.
\begin{figure}[h!t]
\centering
\includegraphics[width=0.4\linewidth]{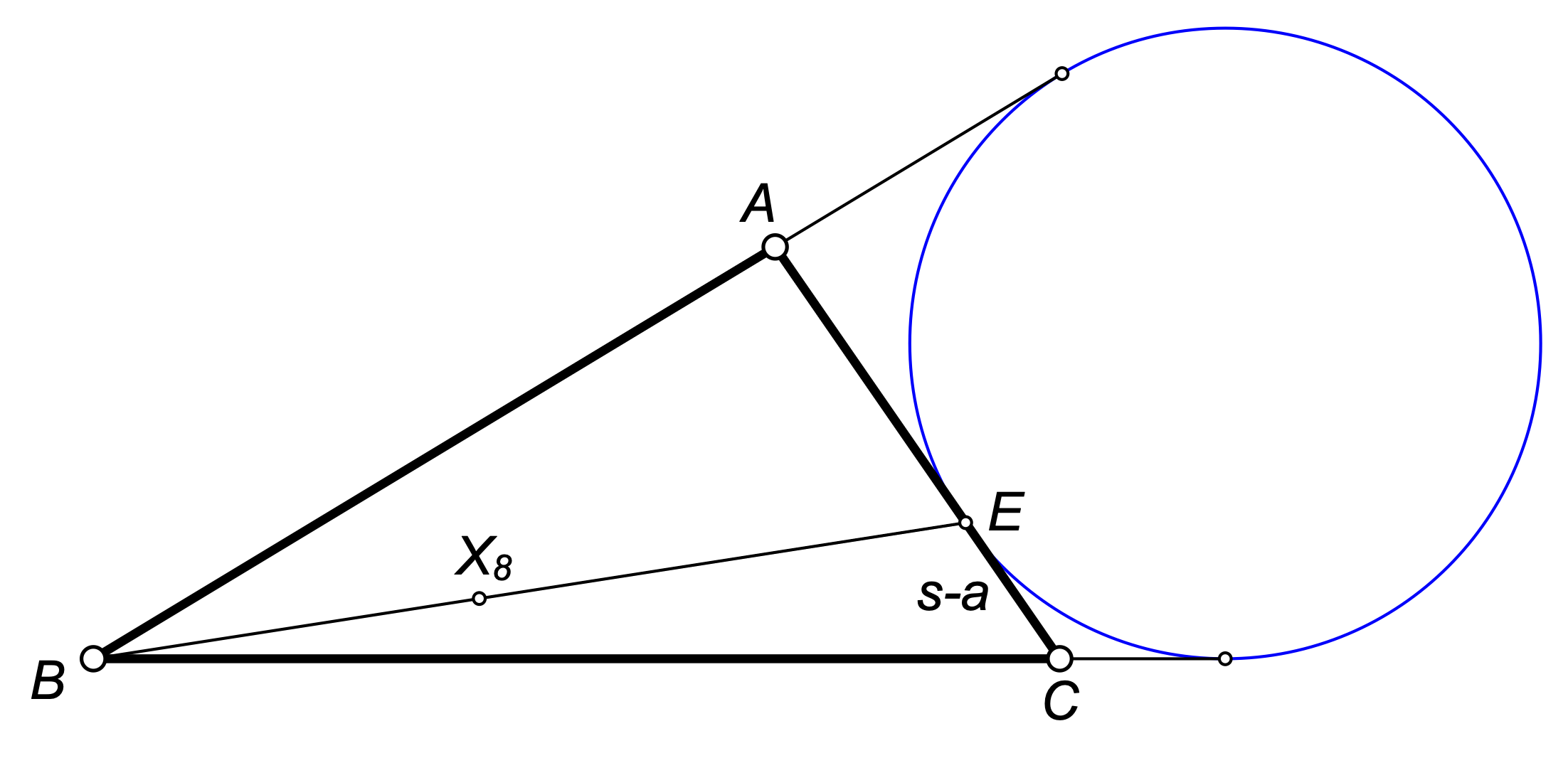}
%\caption{}
\label{fig:B-Y1-X8proof}
\end{figure}

Then $Y_1$ lies on $BE$, if and only if $s-a=u$. Eliminating $u$ between this condition and
Equation~(\ref{eq:u}) gives the condition to be
$a^2(a-b-c)=bc(b+c-3a)$
as required.
\end{proof}

In the same manner, we get the following theorem.

\begin{theorem}
\label{thm:Gergonne}
In $\triangle ABC$ (named counterclockwise), $CY_1$ passes through $X_7$
if and only if $$a^2(a-b-c)=bc(b+c-3a).$$
\end{theorem}

\section{Cevian Triangles}

Let $[XYZ]$ denote the area of $\triangle XYZ$.

\begin{lemma}[The Common Angle Theorem]
In $\triangle ABC$, let $D$ be a point on $BC$ and let $E$ be a point on $AC$ (Figure~\ref{fig:commonAngle}).
Then $$[CDE]=\frac{CE\cdot CD}{CA\cdot CB}[ABC].$$
\end{lemma}

\begin{figure}[h!t]
\centering
\includegraphics[width=0.28\linewidth]{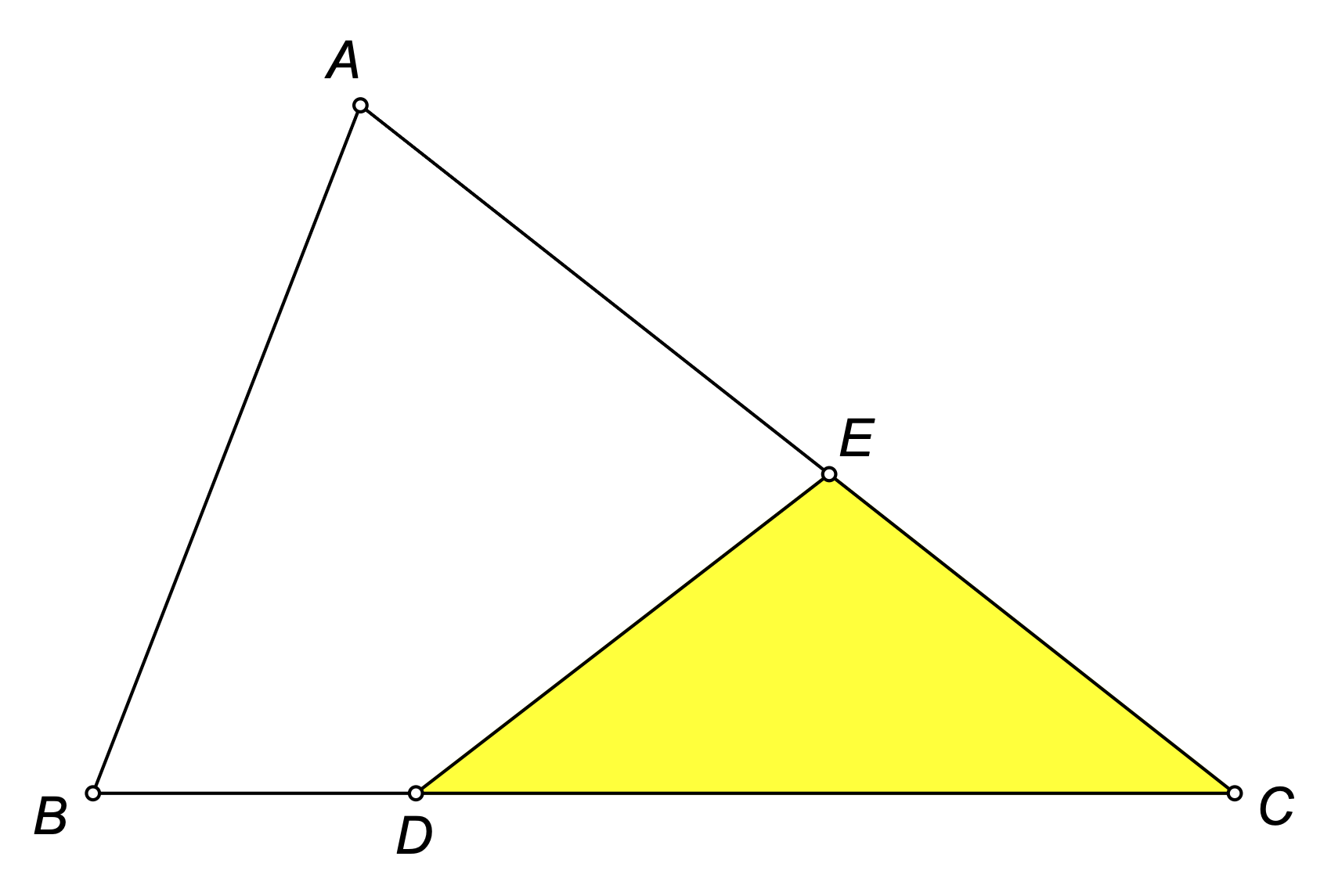}
\caption{}
\label{fig:commonAngle}
\end{figure}

\begin{proof}
This follows from the area formulas $[CDE]=\frac12 CD\cdot CE \sin C$ and
$[ABC]=\frac12 CA\cdot CB \sin C$
\end{proof}

%\newpage

\begin{theorem}[Area of Cevian Triangle]
\label{thm:area}
Let $P$ be a point inside $\triangle ABC$. Let $AP$ meet $BC$ at $D$,
$BP$ meet $CA$ at $E$ and $CP$ meet $AB$ at $F$.
Suppose $[ABC]=K$, $AB=c$, $BC=a$, $CA=b$, $BD=a_1$, $DC=a_2$, $CE=b_1$, $EA=b_2$, $AF=c_1$, and $FB=c_2$ as shown in
Figure~\ref{fig:cevianTriangle}.
Then $$[DEF]=\left(1-\frac{c_1b_2}{bc}-\frac{a_1c_2}{ca}-\frac{b_1a_2}{ac}\right)K.$$
\end{theorem}

Triangle $DEF$ is known as the \emph{cevian triangle} of point $P$.

\begin{figure}[h!t]
\centering
\includegraphics[width=0.4\linewidth]{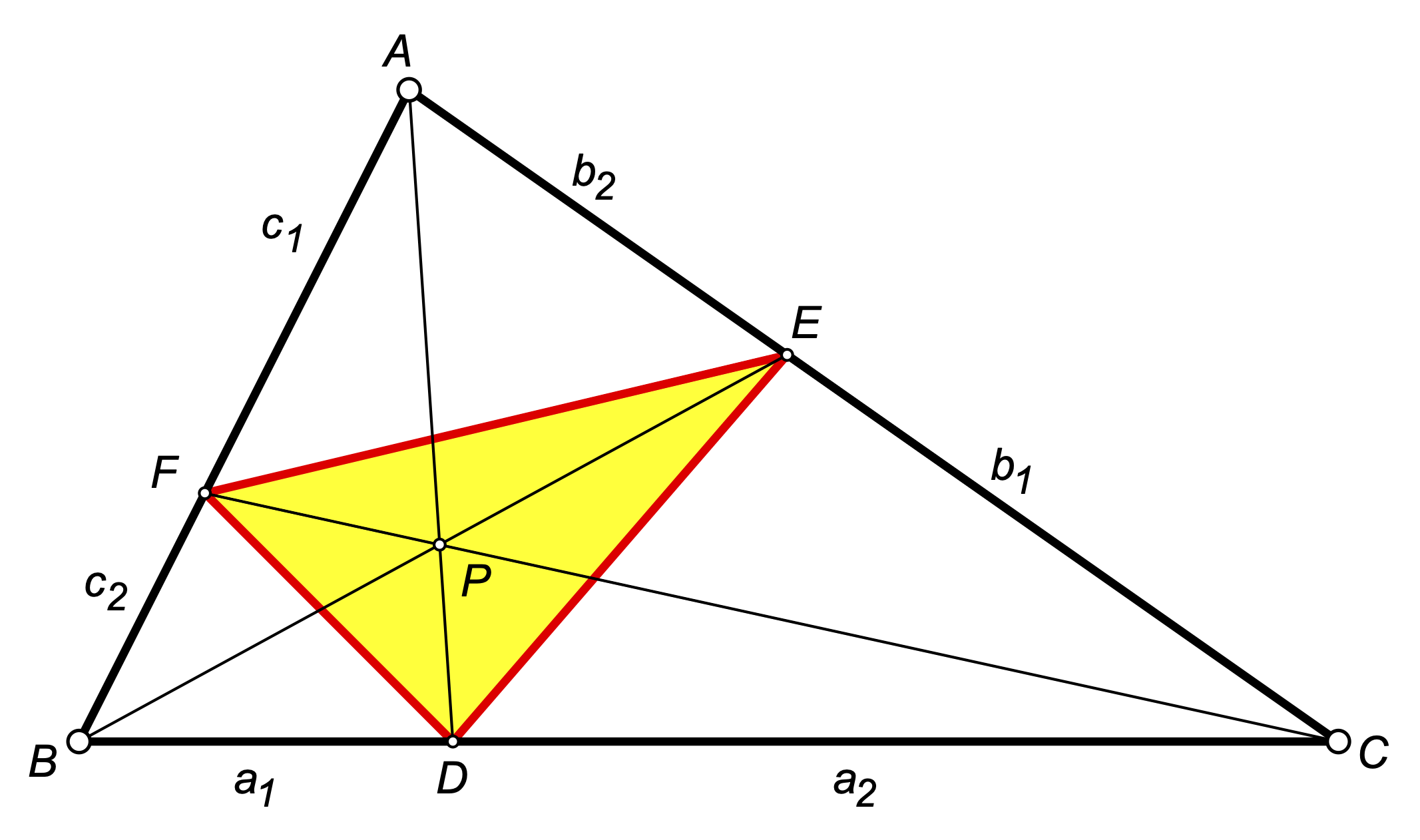}
\caption{}
\label{fig:cevianTriangle}
\end{figure}

\begin{proof}
From the Common Angle Theorem, we have
\begin{align*}
[AFE]&=\frac{c_1b_2}{bc}K,\\
[BDF]&=\frac{a_1c_2}{ca}K,\\
[CED]&=\frac{b_1a_2}{ac}K.
\end{align*}
Since $[DEF]=[ABC]-[AFE]-BDF]-[CED]$, we have
$$[DEF]=\left(1-\frac{c_1b_2}{bc}-\frac{a_1c_2}{ca}-\frac{b_1a_2}{ac}\right)K$$
as required.
\end{proof}

Yff \cite{Yff} gave a formula for the area of the cevian triangle of $Y_1$, but he gave no proof.
We now give a purely geometric proof of his result.

%\newpage

\begin{theorem}[Yff Area Formula]
In $\triangle ABC$, let $AY_1$ meet $BC$ at $D$,
$BY_1$ meet $CA$ at $E$ and $CY_1$ meet $AB$ at $F$ (Figure~\ref{fig:Y1cevianTriangle}).
Then $$[DEF]=\frac{u^3}{2R}.$$
\end{theorem}

\begin{figure}[h!t]
\centering
\includegraphics[width=0.4\linewidth]{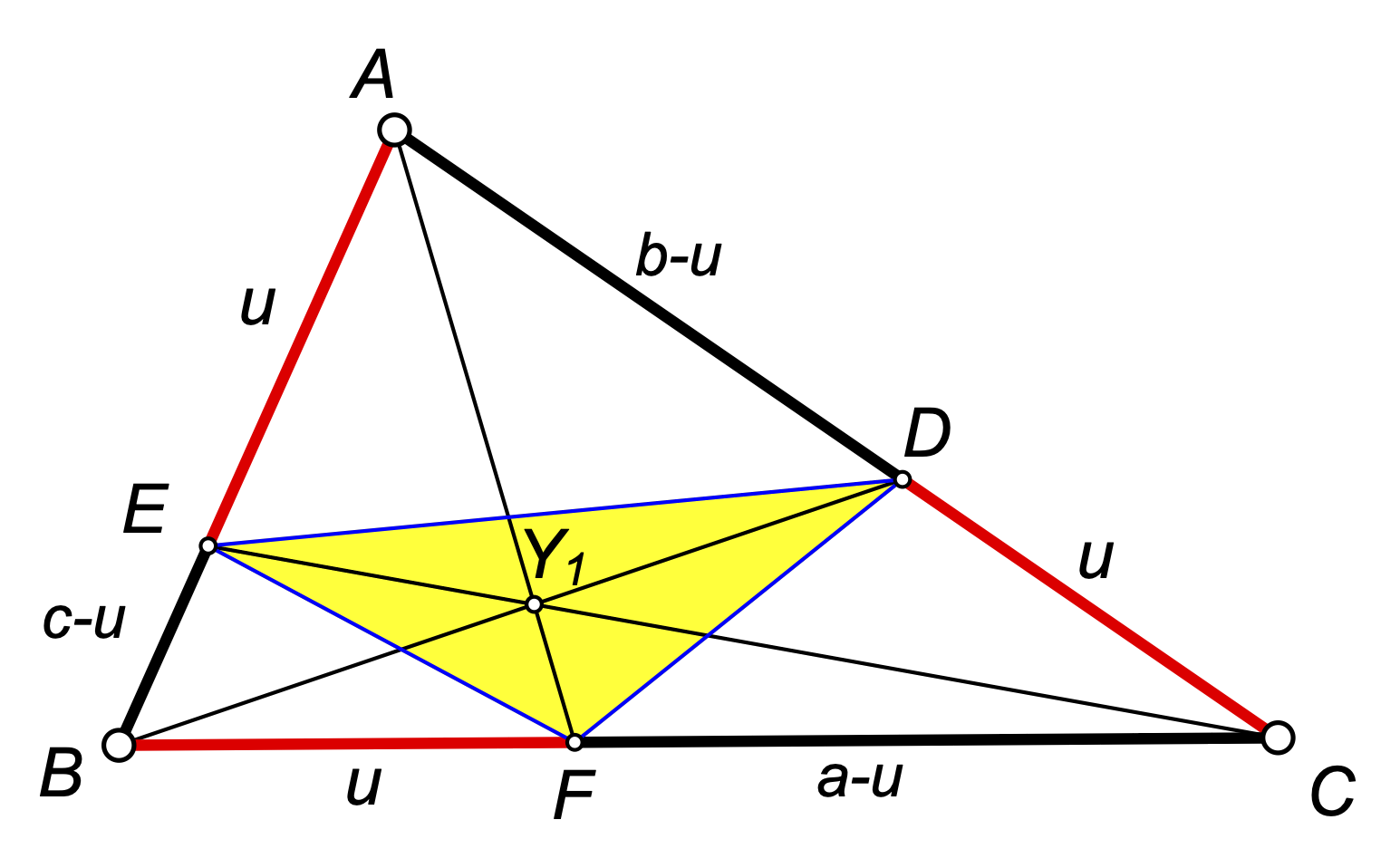}
\caption{}
\label{fig:Y1cevianTriangle}
\end{figure}

\begin{proof}
From the definition of $Y_1$, the segments along the sides of $\triangle ABC$ have the lengths shown in
Figure~\ref{fig:Y1cevianTriangle}.
From Theorem~\ref{thm:area}, we have
\begin{align*}
[DEF]&=\left(1-\frac{u(b-u)}{bc}-\frac{u(c-u)}{ca}-\frac{u(a-u)}{ab}\right)K\\
&=\frac{u^3+(a-u)(b-u)(c-u)}{abc}K\\
&=\frac{u^3+u^3}{abc}K\\
&=2u^3\times \frac{K}{abc}\\
&=\frac{u^3}{2R}
\end{align*}
since $R=abc/(4K)$.
\end{proof}

In a similar manner, we find that the area of the cevian triangle of $Y_2$ is also $u^3/(2R)$.

Weisstein \cite{MathWorld} calls these triangles \emph{Yff Triangles}.

\section{Special Points}

Let $Y_m$ be the midpoint of $Y_1Y_2$.
The point $Y_m$ is a triangle center, but it is not one of the ones cataloged in ETC \cite{ETC}
as of January 1, 2026.
The first normalized trilinear coordinate of $Y_m$ in a 6--9--13 triangle is 0.4500479513210176.

Let $Y_c$ be the center of the unique conic that passes through the points $A$, $B$, $C$, $Y_1$, and $Y_2$.
The point $Y_c$ is not a triangle center listed in ETC \cite{ETC} as of January~1, 2026.
The first normalized trilinear coordinate of $Y_c$ in a 6--9--13 triangle is 0.7440584232553069.

%\newpage

%\printindex
\end{document}